\documentclass{article}

 \usepackage[main, final]{neurips_2026}
\usepackage[utf8]{inputenc} % allow utf-8 input
\usepackage[T1]{fontenc}    % use 8-bit T1 fonts
\usepackage{hyperref}       % hyperlinks
\usepackage{url}            % simple URL typesetting
\usepackage{booktabs}       % professional-quality tables
\usepackage{amsfonts}       % blackboard math symbols
\usepackage{nicefrac}       % compact symbols for 1/2, etc.
\usepackage{microtype}      % microtypography
\usepackage{xcolor}         % colors

\usepackage{microtype}
\usepackage{graphicx}
\usepackage{subcaption}
\usepackage{booktabs} % for professional tables
\usepackage{enumitem}       % custom lists
\usepackage{amsmath}

\usepackage{amsmath}
\usepackage{amssymb}
\usepackage{mathtools}
\usepackage{amsthm}
\RequirePackage{algorithm}
\RequirePackage{algorithmic}
\usepackage{natbib}

\usepackage{array}         
\usepackage{wrapfig}    

\usepackage[capitalize,noabbrev]{cleveref}
\usepackage[textsize=tiny]{todonotes}

\theoremstyle{plain}

\theoremstyle{definition}

\theoremstyle{remark}

\theoremstyle{plain}
\newtheorem{Theorem}{Theorem}[section]

\newtheorem{Corollary}[Theorem]{Corollary}

\newtheorem{Lemma}[Theorem]{Lemma}
\newtheorem{Remark}[Theorem]{Remark}
\newtheorem{Example}[Theorem]{Example}
\newtheorem{Proposition}[Theorem]{Proposition}

\title{Fréchet Regression on the Bures-Wasserstein Manifold}

\author{%
  Duc Toan Nguyen, C\'esar A. Uribe\\
  Department of Electrical and Computer Engineering\\ 
  Ken Kennedy Institute \\
  Rice University, Houston, TX, USA\\
  \texttt{\{duc.toan.nguyen,cauribe\}@rice.edu}
}

\begin{document}

\maketitle

\begin{abstract}
Fréchet regression, or conditional barycenters, is a flexible framework for modeling relationships between covariates (usually Euclidean) and response variables in general metric spaces, e.g., probability distributions or positive-definite matrices. However, in contrast to classical barycenter problems, computing conditional counterparts in many non-Euclidean spaces remains an open challenge, as they yield non-convex optimization problems with an affine structure. 
In this work, we study the existence and computation of conditional barycenters in the space of symmetric positive-definite (SPD) matrices with the Bures-Wasserstein metric, which connects SPD regression to optimal transport, avoids matrix logarithms, and preserves well-posedness under signed weights. We provide a sufficient condition for the existence of a minimizer of the conditional barycenter problem that characterizes the regression range of extrapolation. Moreover, we characterize the optimization landscape, proving that under this condition the objective has no local maxima. Additionally, we develop a projection-free and provably correct algorithm for the approximate computation of first-order stationary points. Finally, we provide a stochastic reformulation that enables the use of off-the-shelf stochastic Riemannian optimization methods for large-scale setups. 
Numerical experiments validate the performance of the proposed methods on regression problems of real-world biological networks and on large-scale synthetic Diffusion Tensor Imaging problems.
\end{abstract}

\vspace{-0.2cm}
\section{Introduction}
\vspace{-0.2cm}
Many modern learning problems require \emph{structured prediction}: the response is not a vector in Euclidean space, but a geometric object constrained by physics, algebra, or topology. A prominent example is the cone of symmetric positive definite (SPD) matrices, which arises as diffusion tensors in medical imaging (e.g., DTI) \cite{pennec2006riemannian,pennec2020manifold}, as covariance/precision matrices in statistics and representation learning \cite{suarez2021tutorial}, in low-rank matrix recovery problems~\cite{thanwerdas2023bures,maunu2023bures}, and as graph Laplacians and related operators in network analysis \cite{zhou2022network,zalles2024optimal,calissano2022graph,severn2021non,severn2022manifold}. In these settings, preserving positive definiteness is not optional: predictions that leave the SPD cone can be mathematically invalid or physically meaningless. This motivates regression methods that respect the geometry of the response space.

Fr\'echet regression, also known as \emph{conditional barycenters}, is the general framework of Petersen and M\"uller~\cite{petersen2019frechet} for metric-space-valued responses with Euclidean predictors. Assume the existence of a joint distribution $(X,Y) \sim \mathcal{F}$, where the predictor $X$ takes values in the Euclidean space $(\mathbb{R}^p, \|\cdot\|_2)$ and the response $Y$ takes values in a potentially non-Euclidean metric space $(\Omega,d)$. The Fr\'echet regression predicts at a query covariate $x$ by minimizing the following conditional expectation:
\begin{align}\label{eq:frechet1}
m_\oplus(x) & = \underset{\omega \in \Omega}{\arg\min} \; \mathbb{E}_{(X,Y)\sim\mathcal{F}}\left[d^2(Y, \omega) \mid X = x \right].
\end{align}
% \vspace{-0.1cm}
Moreover,~\eqref{eq:frechet1} is traditionally solved by minimizing its first-order approximation which is a \emph{weighted Fr\'echet mean} problem, i.e., minimizing a weighted sum of squared distances to $Y$, i.e., 

\begin{align}\label{eq:frechet}
m_G(x) & = \underset{\omega \in \Omega}{\arg\min} \; \mathbb{E}_{(X,Y)\sim\mathcal{F}} \left[ s_G(x) d^2(Y, \omega) \right],
\end{align}
where the weight function is given by $ s_G(x) = 1 + (X - \mu)^\top \Sigma^{-1} (x - \mu) $, with $\mu = \mathbb{E}[X]$ and $\Sigma = \mathrm{Var}(X)$. In practice, given $n$ independent samples $(X_k,Y_k) \sim \mathcal{F}$, $k \in \lbrace 1,...,n \rbrace$, the corresponding empirical estimator of the global function takes the form:
\begin{align} \label{eqn:global-frechet}
    \hat{m}_G(x) =\underset{\omega \in \Omega}{\arg\min} \frac{1}{n} \sum_{k=1}^n s_{G,k}(x) d^2(Y_k, \omega),
\end{align}
where $ s_{G,k}(x) = 1 + (X_k - \bar{X})^\top \hat{\Sigma}^{-1} (x - \bar{X}) $, for $k \in \lbrace 1,...,n \rbrace$, $\bar{X}$ is the sample mean, and $\hat{\Sigma}$ is the sample covariance matrix of $\lbrace X_k \rbrace_{k=1}^n$.

The global Fr\'echet estimator \eqref{eq:frechet} is especially attractive because its weights depend \emph{affinely} on $x$, enabling extrapolation beyond the training covariate range. However, this same feature exposes a fundamental, and often under-emphasized, difficulty: \emph{global Fr\'echet regression inevitably produces negative weights under extrapolation}. This transforms metric regression from a familiar (convex) barycenter computation into a \emph{signed (affine) barycenter} problem on a curved space, where classical well-posedness and computation are no longer guaranteed. To see why negative weights are not uncommon, consider the one-dimensional covariates $X=\{-1,0,1\}$, so $\bar X=0$ and $\hat\Sigma=2/3$. For a query \textit{far} from the mean, e.g., $x=3$, the global weights in \eqref{eqn:global-frechet} satisfy
$(s_{G,1},s_{G,2},s_{G,3})=(-3.5,\,1,\,5.5)$, which includes a negative value. In general, once some $s_{G,k}<0$, the weighted Fr\'echet objective can lose coercivity: minimizers may fail to exist in $\Omega$ and minimizing sequences may drift toward the boundary of the space. For SPD-valued responses, this manifests as a collapse toward singular positive-semidefinite matrices, undermining both the statistical estimator and its computation. Thus, a central question for global Fr\'echet regression is:
\emph{When is the signed barycenter objective well-posed, and how can we compute it without leaving the constraint set?}

\begin{wrapfigure}{r}{0.44\linewidth}
    \vspace{-0.45cm}
    \centering
    \includegraphics[width=\linewidth]{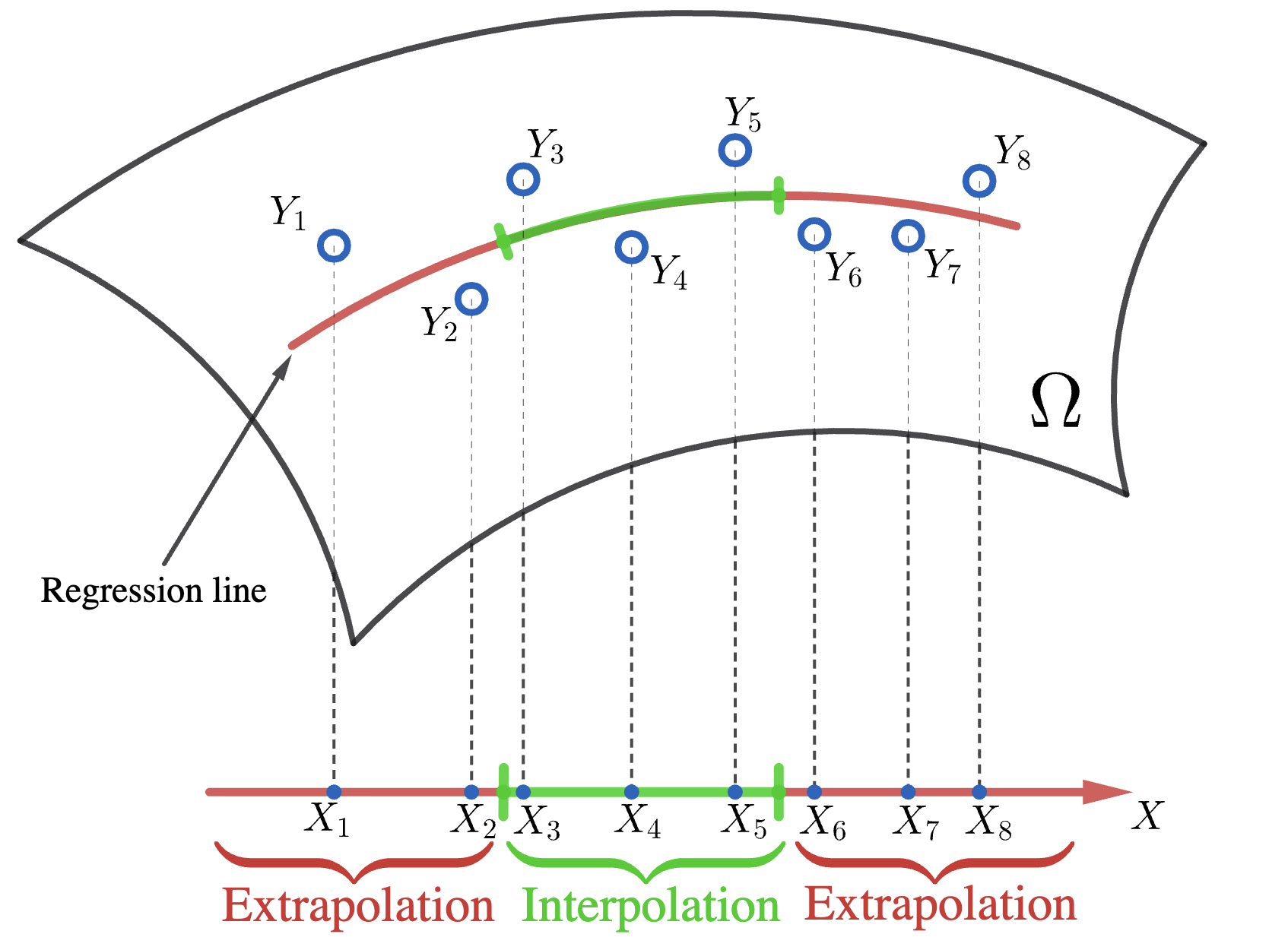}
    \vspace{-0.55cm}
    \caption{Ranges of interpolation (all positive weights) and extrapolation (possibly negative weights) for Fréchet regression with $X_i \in \mathbb{R}$ and $Y_i \in \Omega$, for $i \in \lbrace 1,...,8 \rbrace$.}
    \label{fig:extrapolation}
    \vspace{-0.4cm}
\end{wrapfigure}
% Figure 1: wrapped on the right of the paragraph introducing BW.
In this paper, we answer these questions for Fr\'echet regression on the space of SPD matrices, $\mathbb{S}_{++}^d := \lbrace \Sigma \in \mathbb{R}^{d \times d} \; : \; \Sigma^\top = \Sigma, \Sigma \succ 0 \rbrace$, under the \emph{Bures--Wasserstein} (BW) geometry \cite{bhatia2019bures}.
SPD matrices are fundamental to a wide variety of machine learning applications, serving as the core representation in metric and kernel learning \cite{guillaumin2009you,jawanpuria2015efficient,suarez2021tutorial}, medical imaging \cite{pennec2006riemannian,pennec2020manifold}, computer vision \cite{cherian2016positive}, natural language processing \cite{jawanpuria2019learning}, and network science \cite{haasler2024bures}. 
The SPD cone admits many geometries, none of which is universally preferable (Appendix~\ref{appen:metric-comparison}). For the signed problem, BW offers two complementary advantages. First, BW is the 2-Wasserstein distance between zero-mean Gaussians~\cite{alvarez2016fixed}, which grounds SPD regression in optimal transport and suits the graph-signal covariance representation of networks~\cite{haasler2024bures}. Second, BW avoids matrix logarithms, has a linear rather than quadratic dependence on SPD matrices compared with the Affine-Invariant metric~\cite{pennec2006riemannian,han2021riemannian}, and, unlike Affine-Invariant and Log-Euclidean~\cite{arsigny2006log} distances, remains finite at singular matrices. Hence, signed-weight extrapolation stays stable near the boundary of the cone, where Affine-Invariant and Log-Euclidean predictions can diverge (Appendices~\ref{appen:BW-advantages} and~\ref{appen:BW-AI-LE}). Other metrics may be preferable when affine or scale invariance, or shape preservation, is central. These advantages have led to growing interest in BW-based learning and regression \cite{han2021riemannian,xu2025wasserstein,zalles2024optimal,haasler2024bures}. Yet, existing theory and algorithms largely focus on the \emph{positive-weight} regime (Wasserstein barycenters) \cite{agueh2011barycenters}, while the signed regime required for extrapolation remains much less understood, especially when one must ensure that the solution and all iterates remain in $\mathbb{S}_{++}^d$.

BW Fr\'echet regression leads to a signed objective; when all weights are non-negative, \eqref{eqn:global-frechet} reduces to the classical Wasserstein barycenter problem, for which existence and uniqueness are guaranteed under standard conditions \cite{agueh2011barycenters, kroshnin2021Stat}. When some normalized weight $\lambda_k := s_{G,k}(x)/n$ (so that $\sum_k \lambda_k~=~1$) is negative, the problem becomes a \emph{BW signed barycenter}: the objective is no longer convex in any ambient Euclidean sense, existence can fail, and standard gradient steps can exit $\mathbb{S}_{++}^d$. This signed regime is not merely theoretical: it is the operational regime of extrapolation (Figure~\ref{fig:extrapolation}), which is unavoidable in forecasting tasks. As Table~\ref{tab:regression-methods} in Appendix~\ref{appen:metric-comparison} shows, prior methods avoid this regime through non-negative weights or parametric models, assume that the estimator exists, prove existence only in special configurations, or use flat geometries.
\vspace{-0.2cm}
\paragraph{Our contributions.}
We resolve the open problem of BW Fr\'echet regression under signed weights by combining new geometric existence theory with projection-free optimization methods that preserve positive definiteness:
\begin{enumerate}[ leftmargin=2.0em,left=-1pt,itemsep=1pt, parsep=-1pt, topsep=-0pt, partopsep=-1pt]
    \item \textbf{Existence theory for BW signed barycenters.}
    We provide a verifiable sufficient condition, \textit{Spectral Dominance of Positive Weights}, that guarantees \eqref{eqn:global-frechet} admits a minimizer in $\mathbb{S}_{++}^d$ even when some weights are negative (Theorem~\ref{thm:existence}). This characterizes a concrete ``safe range'' for extrapolation in BW Fr\'echet regression. Under the same condition, we further characterize the landscape and show that the objective admits no local maxima. Moreover, under additional assumptions, we show that the conditional barycenter problem admits a unique solution.

    \item \textbf{Projection-free optimization with guarantees.}
    In the general signed regime, naive Euclidean-based optimization can violate the SPD constraint, and explicit projection onto $\mathbb{S}_{++}^d$ is expensive and numerically unstable. We show that Riemannian gradient descent is \emph{projection-free} under our \emph{Spectral Dominance of Positive Weights} condition, and that a carefully constructed pairwise stochastic reformulation is projection-free under a stronger, pairwise version of this condition, which is the price of a per-iteration cost independent of $n$. In both cases, iterates remain SPD by construction, with sublinear rates to stationary points.

    \item \textbf{Empirical validation on networks and diffusion tensors.}
    We validate the proposed framework on two distinct structured-prediction settings: network regression and DTI-scale synthetic experiments. On temporal ant social networks, BW regression better preserves key topological descriptors (e.g., degree distributions and modularity) than Euclidean (Frobenius), Affine-Invariant, and Log-Euclidean baselines \cite{zhou2022network,zalles2024optimal}. Furthermore, we show that our stochastic Riemannian reformulation scales to a synthetic diffusion-tensor regression problem with $n=100{,}000$ SPD matrices, effectively bypassing the computational bottleneck of the full-batch algorithm while maintaining numerical stability without projections.
\end{enumerate}

Overall, our results make global Fr\'echet regression on the BW manifold \emph{well-posed and computable} in the extrapolation regime, bridging a critical gap between the statistical flexibility of signed-weight regression \cite{petersen2019frechet} and the geometric/algorithmic requirements of SPD-valued learning.
\vspace{-0.2cm}
\section{Related Work}\label{sec:related-work}
\vspace{-0.2cm}
\paragraph{Regression with manifold- and metric-space-valued responses.}
Regression for manifold-valued responses predates the metric-space framework, e.g., kernel regression through weighted Fr\'echet means with non-negative weights~\cite{davis2010population}, geodesic regression~\cite{thomas2013geodesic}, nonlinear kernel regression~\cite{banerjee2015nonlinear,banerjee2016nonlinear}, and, for SPD responses in diffusion tensor imaging, intrinsic and local polynomial regression~\cite{zhu2009intrinsic,yuan2012local}. Petersen and M\"uller~\cite{petersen2019frechet} developed the general global and local-linear Fr\'echet regression framework for metric-space-valued responses with Euclidean predictors. Its global formulation, with possibly negative affine weights, is the basis of our signed objective. In network analysis, \cite{zhou2022network} utilized the framework to model graph Laplacians under the Frobenius metric, and recent works such as \cite{zalles2024optimal} explore Optimal Transport (Wasserstein) geometries for network regression to better capture topological structures. In addition to network analysis, there is a rich body of literature on regressing probability distributions under the Wasserstein metric \cite{peterson2021wasserstein,chen2023wasserstein, ghodrati2022distribution, fan2024conditional,girshfeld2025neural, zhu2025geodesic, xu2025wasserstein}. Beyond classical global and local regression, the Fréchet framework has recently been integrated into machine learning paradigms, including gradient boosting machines \cite{zhou2025frechetboosting} and transfer learning frameworks \cite{zhang2025wasserstein}. Furthermore, deep learning adaptations, such as Deep Fréchet Regression \cite{iao2025deep, kim2025dfnn}, have been proposed to model non-linear relationships. More importantly, even when applying these deep learning frameworks to the Bures-Wasserstein manifold, the output layer typically involves solving a conditional barycenter problem, for which the existence of solutions remains unclear. Table~\ref{tab:regression-methods} (Appendix~\ref{appen:metric-comparison}) summarizes them.

\textbf{Geometries on the SPD cone.} Alternatives to BW include Affine-Invariant~\cite{pennec2006riemannian,moakher2005differential}, Log-Euclidean~\cite{arsigny2006log,arsigny2007geometric}, power-Euclidean~\cite{dryden2009non}, Cholesky-based~\cite{lin2019riemannian,chen2026fast}, Stein~\cite{sra2012new,salehian2013recursive}, scaling-rotation~\cite{jung2015scaling}, and generalized BW~\cite{han2021generalized} geometries. Appendix~\ref{appen:metric-comparison} compares them, including the Procrustes form of BW~\cite{dryden2009non}, shape-separating constructions~\cite{cetingul2012group}, and tangent-space constructions~\cite{bouza2021higher}, which change the regression target.

\textbf{The Bures-Wasserstein Geometry and Optimization.}
The space of SPD matrices endowed with the Bures-Wasserstein metric forms a Riemannian manifold with non-negative curvature \cite{takatsu2010wasserstein, bhatia2019bures, luo2021geometric}. This geometry is intimately linked to Optimal Transport, as the Bures-Wasserstein distance equals the $2$-Wasserstein distance between zero-mean Gaussian distributions with the respective covariance matrices \cite{alvarez2016fixed}. The Wasserstein Barycenter problem, i.e., the Fréchet mean with positive weights, has been extensively studied, with established results on the existence and uniqueness of the minimizer \cite{agueh2011barycenters}. Moreover, several computational approaches have been proposed, including fixed-point \cite{alvarez2016fixed} and Riemannian optimization methods \cite{chewi2020gradient,altschuler2021averaging,weber2023riemannian,junyi2024convergence}.

\textbf{Fréchet Regression on the Bures-Wasserstein Manifold.}
Recently, the statistical properties of Fréchet regression on the Bures-Wasserstein manifold have attracted attention. \cite{xu2025wasserstein} and \cite{kroshnin2021Stat} established asymptotic distributions and F-tests for this setting; however, their analyses assume the existence of a regression estimator without deriving the geometric conditions required for signed weights (extrapolation). From an optimal transport perspective, \cite{fan2024conditional} proved the existence of a distribution regression, but their result is restricted to cases where all sample distributions lie along a single geodesic. Our work relaxes this assumption, proving existence for a much broader open region of the manifold defined by a \emph{Spectral Dominance of Positive Weights} condition. Furthermore, \cite{tornabene2025generalized} analyzed signed barycenters on the general Wasserstein space, but their results do not guarantee that the signed combination of Gaussian distributions remains Gaussian. In the context of SPD matrices, this means their solution could exit the manifold. Our theory explicitly ensures the solution remains a valid SPD matrix. Similarly, \cite{gallouet2025metric} proposed a framework for metric extrapolation, but it is limited to only two distributions. We generalize this to affine combinations of arbitrary $n$ points. Moreover, the proposed method results in the first projection-free Riemannian reformulation, enabling scalable, constraint-preserving optimization for Bures-Wasserstein extrapolation.

\textbf{Existence and Uniqueness of Riemannian Center of Mass.} A minimizer of \eqref{eqn:global-frechet} is called the Riemannian Center of Mass, and commonly studied in subdivision schemes \cite{cavaretta1991stationary,peters2008subdivision,dyn2025subdivision,itai2013subdivision}. A Riemannian center of mass always exists locally, and if the manifold is Cartan-Hadamard with non-positive curvature, it is also unique~\cite{karcher1977riemannian}. We build upon recent results where, for positively curved spaces, explicit bounds for the existence and uniqueness of the Riemannian center of mass have been shown in~\cite{dyer2016barycentric,huning2022convergence}.
\vspace{-0.3cm}
\section{Existence and uniqueness of Bures-Wasserstein conditional barycenters}
\vspace{-0.2cm}
\subsection{Existence of Bures-Wasserstein conditional barycenters}
In this subsection, we propose a condition, called the \textit{Spectral Dominance of Positive Weights}, for the existence of a minimizer of the conditional Bures-Wasserstein barycenter problem. First, we formally restate the main problem with respect to the Bures-Wasserstein metric as follows: Given $n$ SPD matrices $\Sigma_k$ and their weights $\lambda_k \in \mathbb{R}$, for $k \in \lbrace 1,...,n \rbrace$, such that $\sum_{k=1}^n \lambda_k = 1$, we want to solve
\vspace{-0.2cm}
\begin{align}\label{eqn:frechet-obj}
    \underset{S \in \mathbb{S}_{++}^d}{\min} F(S) :=& \sum_{k=1}^n \lambda_k W^2_{2}(S,\Sigma_k) =\sum_{i \in \mathcal{I}} \lambda_i^+ W^2_2(S,\Sigma_i) - \sum_{j \in \mathcal{J}} \lambda_j^- W^2_2(S,\Sigma_j),
\end{align}
\vspace{-0.3cm}

for $\lambda_i^+,\lambda_j^->0, \mathcal{I} = \lbrace k: \lambda_k > 0 \rbrace, \mathcal{J} = \lbrace k: \lambda_k < 0\rbrace, S\in\mathbb{S}_{++}^d$. In the case of all positive weights, the Bures-Wasserstein barycenter problem is shown to have a unique minimizer for any set of SPD matrices \cite{agueh2011barycenters}. However, conditional barycenter problems do not always have solutions in $\mathbb{S}_{++}^d$.

\begin{Example}\label{example}
Let $\lambda_1 = 2$, $\lambda_2=-1$, $\Sigma_1 = \begin{bmatrix}
    1 & 0 \\
    0 & 1 
\end{bmatrix}$,
$\Sigma_2 = \begin{bmatrix}
    9 & 0 \\
    0 & 9
\end{bmatrix}$, and the function $f(S) := \lambda_1 W^2_2(S,\Sigma_1) + \lambda_2 W^2_2(S,\Sigma_2)$. This function has the Euclidean gradient
\begin{align*}
    \nabla f(S) & = I - (\lambda_1 \mathrm{GM}(S^{-1},\Sigma_1) + \lambda_2 \mathrm{GM}(S^{-1},\Sigma_2)) = I + S^{-1/2}.
\end{align*}
For all $S \in \mathbb{S}_{++}^d$, the gradient $\nabla f(S) \succ 0$, which implies that there is no critical point in $\mathbb{S}_{++}^d$.
\end{Example}

Example~\ref{example} shows that we need additional conditions to ensure the existence of a minimizer. Here, we propose a condition that depends on the eigenvalues of all matrices $\Sigma_k$ and their weights $\lambda_k$.

\begin{Theorem}[\textbf{Spectral Dominance of Positive Weights}] \label{thm:existence}
Let $\Sigma_1,{\ldots},\Sigma_n \in \mathbb{S}_{++}^d$ and $\lambda_1,{\ldots},\lambda_n{\in}\mathbb R$ with $\sum_{k=1}^n \lambda_k=1$.
If the {Spectral Dominance of Positive Weights} condition holds, i.e.,
\begin{align}\label{eq:SNM}
\sum_{i \in \mathcal{I}} \lambda_i^+ \sqrt{\lambda_{\min}(\Sigma_i\vphantom{\Sigma_j})} > \sum_{j \in \mathcal{J}} \lambda_j^- \sqrt{\lambda_{\max}(\Sigma_j)},
\end{align}
\vspace{-0.1cm}
then Problem \eqref{eqn:frechet-obj} admits a solution.
\end{Theorem}

The proof is presented in Appendix \ref{appen:proof}. Intuitively, the condition expresses a dominance of the positively weighted matrices. This dominance is quantified by the eigenvalues, which capture the strongest and weakest directions of each matrix. By comparing these extreme directions, the condition ensures that positive contributions outweigh negative ones in every possible direction. Since the function $F(S)$ is differentiable on $\mathbb{S}_{++}^d$, under condition (\ref{eq:SNM}), there is at least one stationary point $S_\ast$ that satisfies

\vspace{-0.5cm}
\begin{align} \label{eqn:stationary-equation}
    S_* = \sum_k \lambda_k \left( S_*^{1/2} \Sigma_k S_*^{1/2} \right)^{1/2}.
\end{align}
\vspace{-0.5cm}

Now, we have some characteristics of stationary points. First, we show that all stationary points stay on a bounded subset of $\mathbb{S}_{++}^d$.
\begin{Proposition}\label{prop:stationary-bound}
If the condition in Theorem \ref{thm:existence} holds, then any stationary point $S_*$ satisfies:
\begin{align*}
    \Big( \sum_{i\in \mathcal{I}} \lambda_i^+ \hspace{-0.05cm} \sqrt{\lambda_{\min}(\Sigma_i\vphantom{\Sigma_j})} {-} \sum_{j \in \mathcal{J}} \lambda_j^- \hspace{-0.05cm}\sqrt{\lambda_{\max}(\Sigma_j)} \Big)^2 \hspace{-0.05cm}I {\preceq} S_* {\preceq}\Big( \sum_{i\in \mathcal{I}} \lambda_i^+ \hspace{-0.05cm} \sqrt{\lambda_{\max}(\Sigma_i\vphantom{\Sigma_j})} {-} \sum_{j\in \mathcal{J}} \lambda_j^- \hspace{-0.05cm}\sqrt{\lambda_{\min}(\Sigma_j)} \Big)^2 \hspace{-0.05cm} I.
\end{align*}
\end{Proposition}
\vspace{-0.3cm}
The proof is presented in Appendix~\ref{appen:proof}. Next, we will show that there is no local maximum.

\begin{Proposition}\label{prop:no-maximum}
Assume that $F(S)$ has a stationary point $S_* {\in} \mathbb{S}_{++}^d$. Then, $S_*$ is not a local maximum.    
\end{Proposition}

The proof is presented in Appendix \ref{appen:proof}. Thus, Proposition~\ref{prop:no-maximum} and condition \eqref{eq:SNM} ensure that the objective function only has local minima or saddles.

\subsection{Uniqueness of Bures-Wasserstein conditional barycenters}
Following \cite[Theorem 3.4.9]{wintraecken2015ambient}, we can guarantee the uniqueness of the minimizer for Problem~\eqref{eqn:frechet-obj} under the assumption that all given matrices $\Sigma_k$ stay in $B_r(\Sigma_0)$, which is a geodesic ball centered at $\Sigma_0$ with radius $r$ depending on the curvature of the manifold and the negative-positive weight ratio. Let us first introduce some auxiliary geometric properties of the Bures-Wasserstein manifold.

\begin{Lemma} \label{prop:upper-bound-curvature}
Let $\Sigma \in \mathbb{S}_{++}^d$ with its smallest eigenvalue $\lambda_{\min}(\Sigma)$, then all sectional curvatures at point $\Sigma$ are bounded above by ${3}/{(2 \lambda_{\min}(\Sigma))}$.
\end{Lemma}

The proof of Lemma~\ref{prop:upper-bound-curvature} is given in Appendix \ref{appen:proof}. Now, given $n$ matrices $\{\Sigma_k\}_{k=1}^n$ in Problem~\eqref{eqn:frechet-obj}, let $0< \lambda < \min_k \, \lambda_{\min}(\Sigma_k)$, we define the set 
$\mathcal{S}_\lambda = \lbrace \Sigma \in \mathbb{S}_{++}^d \,: \, \lambda_{\min} (\Sigma) \geq \lambda \rbrace.$
From Lemma \ref{prop:upper-bound-curvature}, we have $\mathcal{S}_\lambda$ is a sub-manifold of $\mathbb{S}_{++}^d$ with bounded sectional curvatures $0~\leq~ K(\mathcal{S}_\lambda)~\leq~\Lambda^+, \text{ for } \Lambda^+ = {3}/{(2\lambda)}.$

\begin{Lemma}\label{prop:inj}
    The injectivity radius, i.e., the maximal radius of the ball on which $\mathrm{Exp}_\Sigma$ is a diffeomorphism onto its image, of the set $\mathcal{S}_\lambda$ is given by $\mathrm{inj}(\mathcal{S}_\lambda) := \underset{\Sigma \in \mathcal{S}_\lambda}{\inf} \mathrm{inj}(\Sigma) = \sqrt{\lambda}$.
\end{Lemma}

The proof of Lemma \ref{prop:inj} follows from \citep[Theorem 6]{luo2021geometric}. We are ready to state our result on the uniqueness of the solution of Problem \eqref{eqn:frechet-obj}, following the approach in \citep[Theorem 3.4.9]{wintraecken2015ambient}.

\begin{Proposition}[Unique existence of minimizer] \label{thm:uniqueness} Let $\Sigma_k \in \mathbb{S}_{++}^d$, for $k \in \lbrace 1,...,n \rbrace$, $0 < \lambda < \min \lambda_{\min}(\Sigma_k)$ and $\mathcal{S}_\lambda = \lbrace \Sigma \in \mathbb{S}_{++}^d \,\,: \,\, \lambda_{\min} (\Sigma) \geq \lambda \rbrace.$
Denote $\mu_+ = \sum_i \lambda_i^+, \mu_- = \sum_j \lambda_j^-.$
Assume that there exist $\rho>0$, $r>0$ and $\Sigma_0 \in \mathcal{S}_\lambda$ such that
\begin{itemize}[ leftmargin=2.0em,left=-1pt,itemsep=1pt, parsep=-1pt, topsep=-0pt, partopsep=-1pt]
    \item $\Sigma_k \in B_r(\Sigma_0)$ for $k \in \lbrace 1,...,n \rbrace$,
    \item $r < \rho/(\mu_+ + \mu_-)$, $\rho < \sqrt{\lambda}/2$, $B_\rho(\Sigma_0) \subset \mathcal{S}_\lambda$, and
    \item $\mu_+ / \mu_- > (2\rho \sqrt{\Lambda^+})/(\tanh (2\rho \sqrt{\Lambda^+})).$
\end{itemize}
Then, Problem~\eqref{eqn:frechet-obj} has a unique minimizer in $B_\rho (\Sigma_0)$.
\end{Proposition}

\begin{Remark}\label{rem:uniqueness-vs-existence}
    The conditions in Proposition \ref{thm:uniqueness} are stronger than the condition in Theorem~\ref{thm:existence}. In Proposition \ref{thm:uniqueness}, all data points $\Sigma_k$ are required to be in $B_r(\Sigma_0)$, which restricts the eigenvalues of all matrices to a bounded range determined by the curvature of the manifold and the weight ratio $\mu_+/\mu_-$, while Theorem \ref{thm:existence} allows each matrix $\Sigma_i$ with positive weight $\lambda_i^+$ to have arbitrarily large eigenvalues, for $i \in \mathcal{I}$.
\end{Remark}

\section{Riemannian gradient-based methods}\label{sec:RGD}
\subsection{Bures-Wasserstein Conditional Barycenter Gradient Method}

The Riemannian manifold structure of the Bures-Wasserstein manifold allows Riemannian optimization methods to solve the Fréchet regression Problem~\eqref{eqn:frechet-obj}. In this subsection, we use the standard first-order Riemannian gradient descent (RGD) algorithm (Algorithm \ref{alg:bwgd}). It is exactly RGD on $\Phi := F/2$, which has the same minimizer, which is derived in Appendix~\ref{appen:prelim}.

In \cite{chewi2020gradient,altschuler2021averaging, brahmachari2025fixed}, the authors showed the convergence of the Riemannian gradient descent for the Bures-Wasserstein barycenter problem, i.e., positive weights only, with a constant stepsize $\eta=1$. In  Theorem~\ref{thm:convergence-rate-bwgd}, we show that Riemannian gradient descent also converges to a stationary point of Problem~\eqref{eqn:frechet-obj} with a step size $\eta \leq 1/L_{m_0}$, where $m_0$ is a lower eigenvalue bound on a sublevel set of the objective. 

\begin{algorithm}[H]
\caption{General BW Barycenter Gradient Descent}\label{alg:bwgd}
\begin{algorithmic}[1]
\STATE \textbf{Input}: SPD matrices \(\Sigma_k\), weights \(\lambda_k\), initial \(S_0\), step-size \(\eta\), epochs \(T\).
\FOR{$t = 1, 2, \ldots, T$}
    \STATE $\tilde{S}_t = (1-\eta) I + \eta \sum_{k=1}^n \lambda_k \, \mathrm{GM}(S_{t-1}^{-1},\Sigma_k)$ \hspace{0.2cm} \textcolor{blue}{// $I + \mathcal{L}_{S_{t-1}}[-\eta\nabla_{\mathrm{bw}}\Phi(S_{t-1})]$}
    \STATE $S_t = \tilde{S}_t S_{t-1} \tilde{S}_t$ \hspace{0.3cm} \textcolor{blue}{// $\mathrm{Exp}_{S_{t-1}}(-\eta\nabla_{\mathrm{bw}}\Phi(S_{t-1}))$}
\ENDFOR
\STATE \textbf{Return} \(\{S_0,\cdots, S_T\}\)
\end{algorithmic}
\end{algorithm}
\vspace{-0.4cm}
Before applying Algorithm \ref{alg:bwgd} to Problem~\eqref{eqn:frechet-obj}, we need to ensure the existence of a minimizer first by assuming that the \textit{Spectral Dominance of Positive Weights} condition in Theorem \ref{thm:existence} holds. Moreover, under this assumption, given any initial matrix $S_0 \in \mathbb{S}_{++}^d$, all iterated matrices $S_t$ stay inside the domain $\mathbb{S}_{++}^d$, for all $t \in \lbrace 0, ..., T \rbrace$. 

\begin{Proposition} \label{prop:non-projection-iteration}
Let $\Sigma_k \in \mathbb{S}_{++}^d$ and corresponding weights $\lambda_k \in \mathbb{R}$, for $k \in \lbrace 1,...,n \rbrace$, such that the \textit{Spectral Dominance of Positive Weights} in Theorem \ref{thm:existence} holds. Then, for any $S \in \mathbb{S}_{++}^d$, $\sum_{k=1}^n \lambda_k \mathrm{GM}(S^{-1}, \Sigma_k) \in \mathbb{S}_{++}^d.$
\end{Proposition}

The proof of Proposition~\ref{prop:non-projection-iteration} is presented in Appendix~\ref{appen:proof}.  Proposition~\ref{prop:non-projection-iteration} implies Algorithm \ref{alg:bwgd} is projection-free, since $\tilde S_t \succ 0$ whenever $\eta \leq 1$. Next, we prove the sublinear convergence rate of Algorithm~\ref{alg:bwgd}. First, recall the $1-$smoothness of the Bures-Wasserstein distance function~~\citep[Theorem 7]{chewi2020gradient}, which handles the positively weighted terms; the negatively weighted terms require a lower bound.
\begin{Lemma}\label{lemma:upper}
Let $X,Y, \Sigma \in \mathbb{S}_{++}^d$. Define $G(X): = \frac{1}{2}W^2_2(X,\Sigma)$. Then, $G(Y) \leq G(X) + \langle \nabla_{\mathrm{bw}} G(X), \log_X Y \rangle_X + \frac{1}{2} W^2_2(X,Y).$
\end{Lemma}

\begin{Lemma}\label{lemma:lower}
Under the notation of Lemma~\ref{lemma:upper}, let $T := \mathrm{GM}(X^{-1},\Sigma)$. Then, $\lambda_{\max}(T) \leq \sqrt{\lambda_{\max}(\Sigma)/\lambda_{\min}(X)}$ and
\begin{align*}
    -G(Y) \leq -G(X) + \langle -\nabla_{\mathrm{bw}} G(X), \log_X Y \rangle_X + \tfrac{1}{2}\big(\lambda_{\max}(T)-1\big)_+ W^2_2(X,Y),
\end{align*}
where $(a)_+ := \max\lbrace a,0 \rbrace$.
\end{Lemma}
The constant depends on $X$, since $-G$ is not $1$-smooth on all of $\mathbb{S}_{++}^d$ (the Hessian of $G$ can fall below $-1$ near the boundary). Hence, $\Phi$ is smooth locally:
\begin{Lemma}\label{lemma:lsmooth}
For $m>0$, let $L_m := \mu_+ + \sum_{j \in \mathcal{J}} \lambda_j^-\big(\sqrt{\lambda_{\max}(\Sigma_j)/m} - 1\big)_+ \geq 1$. Then, for all $X,Y \in \mathbb{S}_{++}^d$ with $\lambda_{\min}(X) \geq m$, the function $\Phi = \frac{1}{2}F$ satisfies 
$$\Phi(Y) \leq \Phi(X) + \langle \nabla_{\mathrm{bw}} \Phi(X), \log_X Y \rangle_X + \frac{L_m}{2} W^2_2(X,Y).$$ Without negative weights, $L_m = 1$.
\end{Lemma}

The proofs are given in Appendix~\ref{appen:proof}. Lemma~\ref{lemma:lsmooth} allows us to establish the convergence rate for Algorithm~\ref{alg:bwgd} as follows.

\begin{Theorem}(Convergence rate of RGD) \label{thm:convergence-rate-bwgd}
Let the conditions of Theorem~\ref{thm:existence} hold and $T>0$. Let $F_\ast$ be the minimum value of Problem~\eqref{eqn:frechet-obj} and $F_\partial := \min\lbrace F(X) : X \in \mathbb{S}_+^d,\ \det X = 0\rbrace$. If $S_0\in  \mathbb{S}_{++}^d$ satisfies $F(S_0) < F_\partial$, then $m_0 := \min\lbrace \lambda_{\min}(X) : X \in \mathbb{S}_+^d,\ F(X) \leq F(S_0) \rbrace > 0$, and Algorithm~\ref{alg:bwgd} with $\eta \leq 1/L_{m_0}$ satisfies $\lambda_{\min}(S_t) \geq m_0$ for all $t$ and

\vspace{-0.4cm}

$$\frac{1}{T}\sum_{t=0}^{T-1} \lVert \nabla_{\mathrm{bw}} \Phi(S_{t}) \rVert^2_{S_t} \leq \frac{F(S_0) - F_\ast}{\eta T}, \;\; \text{i.e.,} \;\; \frac{1}{T}\sum_{t=0}^{T-1} \lVert \nabla_{\mathrm{bw}} F(S_{t}) \rVert^2_{S_t} \leq \frac{4(F(S_0) - F_\ast)}{\eta T}.$$

\vspace{-0.3cm}

\end{Theorem}

The proof of Theorem~\ref{thm:convergence-rate-bwgd} is presented in Appendix~\ref{appen:proof}. Here, $F_\partial > F_\ast$ is well defined because $F$ is coercive on $\mathbb{S}_+^d$ and, under~\eqref{eq:SNM}, no singular matrix minimizes $F$ (proof of Theorem~\ref{thm:existence}); hence, any $S_0$ close enough to a minimizer qualifies. With positive weights only, $L_m = 1$ for every $m$, so no condition on $S_0$ is needed, as in~\cite{chewi2020gradient}.

\subsection{Pairwise Riemannian Stochastic Gradient Descent methods for large-scale Bures-Wasserstein conditional barycenters}

Each iteration of the full-batch Algorithm~\ref{alg:bwgd} requires $n$ matrix square roots and inversions, which is costly when $n$ is very large. To tackle this problem, we adopt variations of Riemannian Stochastic Gradient Descent (R-SGD) for finite-sum functions. Recall that our original objective function $F(S) = \sum_{k} \lambda_k W^2_2(S,\Sigma_k)$ is a weighted sum of distance functions $W^2_2(S,\Sigma_k)$, where the weights $\lambda_k$ can be negative. If we iterate one R-SGD step with respect to a distance function with negative weight, the output of that step can be out of the domain $\mathbb{S}_{++}^d$, and we need one extra projection step. Therefore, we propose a reformulation of the objective function to address the difficulty posed by negative weights as follows. Following Problem~\eqref{eqn:frechet-obj}, assume that there is at least one negative weight; then, without loss of generality, we can reformulate the cost function as:
\begin{align*}
    F(S) =\sum_{i \in \mathcal{I}} \lambda_i^+ W^2_2(S,\Sigma_i) {-} \sum_{j \in \mathcal{J}} \lambda_j^- W^2_2(S,\Sigma_j) =\sum_{i,j}\frac{\lambda_i^+\lambda_j^-}{\mu_+\mu_-}\underbrace{\left(  \mu_+ W^2_2(S, \Sigma_i) {-} \mu_- W^2_2(S,\Sigma_j)\right)}_{f_{ij}(S)}, 
\end{align*}
for $\lambda_i^+,\lambda_j^->0, \mu_+ := \sum_i \lambda_i^+, \mu_- := \sum_j \lambda_j^-, S\in\mathbb{S}_{++}^d$. In this reformulation, each elementary function $f_{ij}$ contains one matrix $\Sigma_i$ with positive weight $\lambda_i^+$ and one matrix $\Sigma_j$ with negative weight $-\lambda_j^-$. Moreover, from this reformulation, we can use some off-the-shelf Stochastic Riemannian Optimization algorithms,
such as R-SGD \cite{bonnabel2013stochastic}, R-SVRG \cite{zhang2016riemannian}, R-SRG \cite{kasai2018riemannian}, or R-SPIDER \cite{zhang2018r, zhou2019faster},
where each unbiased stochastic gradient of function $f_{ij}$ is
$$\nabla f_{ij}(S) = I - \left( \mu_+ \mathrm{GM}(S^{-1}, \Sigma_i) - \mu_- \mathrm{GM}(S^{-1},\Sigma_j) \right),$$
for $S {\in} \mathbb{S}_{++}^d$. Here, we present an example of Riemannian Stochastic Gradient Descent as Algorithm~\ref{alg:rsgd}, which, as Algorithm~\ref{alg:bwgd}, performs Riemannian steps on the scaled components $\phi_{ij} := \frac{1}{2}f_{ij}$. 

\begin{algorithm}[H]
\caption{Pairwise Riemannian SGD}\label{alg:rsgd}
\begin{algorithmic}[1]
\STATE \textbf{Input}: SPD matrices \(\Sigma_k\), weights \(\lambda_k\), initial \(S_0\), step-size $\eta_t$, epochs \(T\).
\FOR{$t = 1, 2, \ldots, T$}
    \STATE Sample $i~{\in}~\mathcal{I}$ w.p. $\lambda_i^+ / \mu_+$, $j~{\in}~\mathcal{J} $ w.p. $\lambda_j^- / \mu_-$. 
    \STATE $\tilde{S}_t = (1-\eta_t) I + \,\eta_t \left( \mu_+ \mathrm{GM}(S_{t-1}^{-1}, \Sigma_i) - \mu_- \mathrm{GM}(S_{t-1}^{-1},\Sigma_j) \right)$ \hspace{0.1cm}
    \STATE $S_t = \tilde{S}_t S_{t-1} \tilde{S}_t$ \hspace{0.3cm} \textcolor{blue}{// $\mathrm{Exp}_{S_{t-1}}(-\eta_t\nabla_{\mathrm{bw}}\phi_{ij}(S_{t-1}))$}
\ENDFOR
\STATE \textbf{Return} \(\{S_0,\cdots, S_T\}\)
\end{algorithmic}
\end{algorithm}
\vspace{-0.3cm}
As in the last subsection, before running any optimization algorithms, we need to ensure the existence of a minimizer of Problem~\eqref{eqn:frechet-obj}. However, for the stochastic reformulation, we use a stronger condition instead of the original condition~\eqref{eq:SNM}, which is
\begin{align} \label{eqn:Stricter-existence-condition}
    (\mu_+) \underset{i \in \mathcal{I}}{\min} \sqrt{\lambda_{\min}(\Sigma_i\vphantom{\Sigma_j})} > (\mu_-) \underset{j \in \mathcal{J}}{\max} \sqrt{\lambda_{\max}(\Sigma_j)}
\end{align}
Under this condition, any iteration using the stochastic gradient $\nabla f_{ij}$ will produce a new matrix $S_t$ that remains inside the domain $\mathbb{S}_{++}^d$ without any projection, for any $t$. This is the same as the property of Algorithm $\ref{alg:bwgd}$, which is proved in Proposition \ref{prop:non-projection-iteration}.

\begin{Corollary}\label{col:projection-free-stochastic}
    Let $\Sigma_k \in \mathbb{S}_{++}^d$ and corresponding weights $\lambda_k \in \mathbb{R}$, for $k \in \lbrace 1,...,n \rbrace$, and let~\eqref{eqn:Stricter-existence-condition} hold. Then, for any $S \in \mathbb{S}_{++}^d$, $ i~\in~\mathcal{I}$, and $j~\in~\mathcal{J}$, it follows that
    $$\mu_+ \mathrm{GM}(S^{-1},\Sigma_i) - \mu_- \mathrm{GM}(S^{-1},\Sigma_j) \in \mathbb{S}_{++}^d.$$
\end{Corollary}

\textbf{Price of the pairwise guarantee.} Condition~\eqref{eqn:Stricter-existence-condition} is needed only for Algorithm~\ref{alg:rsgd}, since every sampled pair must stay SPD without help from the other matrices. With $a_i := \sqrt{\lambda_{\min}(\Sigma_i)}$, $b_j := \sqrt{\lambda_{\max}(\Sigma_j)}$, weighted means $\bar a, \bar b$, and $q := \mu_-/\mu_+$, \eqref{eq:SNM} certifies $q < \bar a/\bar b$ and~\eqref{eqn:Stricter-existence-condition} certifies $q < \min_i a_i/\max_j b_j$, so $\kappa := (\min_i a_i/\max_j b_j)/(\bar a/\bar b) \leq 1$ is the retained fraction, which one near-singular positively weighted matrix can make small. If only~\eqref{eq:SNM} holds, the problem is still well posed and Algorithm~\ref{alg:bwgd} applies.

Riemannian Stochastic Gradient Descent variations have been widely investigated under different sets of assumptions, summarized in \citep[Table 1]{oowada2025faster}. For completeness, we present the convergence rate analysis of Algorithm~\ref{alg:rsgd} following~\cite{oowada2025faster}. 

\begin{Theorem} \label{thm:RSGD-convergence}
Let condition~\eqref{eqn:Stricter-existence-condition} hold and $T>0$. Assume that there exists $ 0 \leq \sigma < \infty$ such that $\mathbb{E}[\lVert \nabla_{\mathrm{bw}}f_{ij}(S) {-} \nabla_{\mathrm{bw}}F(S) \rVert_S^2] \leq \sigma^2 $ for all $S \in \mathbb{S}_{++}^d$, and that the iterates satisfy $\lambda_{\min}(S_t) {\geq} m {>} 0$, so that each $\phi_{ij}$ satisfies Lemma~\ref{lemma:lsmooth} with $\hat L_m := \mu_+ + \mu_-(\sqrt{\max_{j}\lambda_{\max}(\Sigma_j)/m} - 1)_+$. Then, the sequence $\{S_t\}_{t=0}^{T-1}$ generated by Algorithm~\ref{alg:rsgd} with $\eta_t = \eta_{0}/\sqrt{t+1}$, $\eta_0 \leq 1$, for $t \in \lbrace 1,...,T \rbrace$, has the following property:
$$\underset{t \in \lbrace 0,...,T-1 \rbrace}{\min} \mathbb{E}[\lVert \nabla_{\mathrm{bw}}F(S_t) \rVert_{S_t}^2] \leq \frac{Q_1+Q_2 \sigma^2 \log T}{\sqrt{T}},$$
where $Q_1, Q_2$ are constants independent of $T$.
\end{Theorem}
Unlike in Theorem~\ref{thm:convergence-rate-bwgd}, stochastic steps need not decrease $F$, so the sublevel-set argument does not apply and we keep this eigenvalue assumption. The projection-free guarantee of Corollary~\ref{col:projection-free-stochastic} does not require it.

Moreover, Riemannian Stochastic Variance-Reduced algorithms can theoretically establish better convergence rates and lower total number of gradient computations (i.e., total complexities), which are summarized in \cite{zhang2018r,zhou2019faster}. However, the trade-off is that some iterations of these algorithms require full-batch gradient computation, which poses a significant computational challenge compared to R-SGD.
\section{Numerical Analysis}\label{sec:num-experiment}
\vspace{-0.2cm}
\subsection{Network regression on an Ant Social Organization} \label{subsec:ant-experiment}
\vspace{-0.2cm}
In this subsection, we show the performance of Algorithm~\ref{alg:bwgd} in network regression on the Ant Social Organization dataset~\citep{mersch2013tracking}. More specifically, each network is represented by its Laplacian~$L$, which is a positive semi-definite matrix. To apply our algorithm, we need to modify all the Laplacians as
$\Sigma := L^\dagger + \frac{1}{d} \mathbf{1}\mathbf{1}^\top = (L + \frac{1}{d}\mathbf{1}\mathbf{1}^\top)^{-1}$, which is SPD because every network considered is connected (Appendix~\ref{appen:condition-validation}). Then, the network regression problem becomes the Fréchet regression problem, following \cite{haasler2024bures}. We source the data from the Network Repository \cite{nr}, which was originally collected by \cite{mersch2013tracking}. The original collection tracks spatial interactions within six separate colonies of ants (Camponotus fellah) over 41 days. For this experiment, we isolate the temporal network of the first colony and select the first 11 days to ensure all networks are connected. Consequently, our final dataset comprises 11 networks, each with 113 nodes representing ants in the first colony. We frame the problem as a regression task where the covariate is the time index $X = \tau$ (ranging over the 11 days) and the response is the corresponding modified Laplacian. The interpolation range, where all weights are positive, is from day $4$ to day $8$, which is less than 50\% of the regression range. The verification for the Spectral Dominance condition is presented in Appendix~\ref{appen:condition-validation}: for this representation, Condition~\eqref{eq:SNM} holds on the interpolation days and fails on the extrapolation days. We apply Algorithm~\ref{alg:bwgd} with an initial point $S_0 = I_{d \times d}$, a step-size $\eta = 1/\sum_k |\lambda_k|$, and a maximum iteration count of $T = 100$. 

First, we evaluate the performance of Algorithm~\ref{alg:bwgd} by computing the Euclidean gradient norm $\lVert \nabla F (S) \rVert_\mathrm{F}$ throughout 100 iterations for all days $\tau$ from 1 to 11. As shown in Figure~\ref{fig:BWGD-grad-norm}, Algorithm~\ref{alg:bwgd} achieves convergence across all 11 days, with the gradient norm $\lVert \nabla F(S)\rVert_{\mathrm{F}}$ vanishing to $10^{-10}$. A clear distinction in rates is visible: interpolation tasks ($\tau \in [4, 8]$) converge in under 10 iterations, while extrapolation tasks (near $\tau=1$ and $\tau=11$) require more steps. This slowdown at the boundaries is expected; extrapolation induces larger negative weights, which complicates the geometry of the objective function compared to the standard all-positive barycenter setting. On the interpolation days, \eqref{eq:SNM} holds and Theorem~\ref{thm:existence}, Proposition~\ref{prop:non-projection-iteration}, and Theorem~\ref{thm:convergence-rate-bwgd} apply. On the extrapolation days, our theory gives no certificate, as~\eqref{eq:SNM} is only sufficient. The iterates and factors $\tilde{S}_t$ remained SPD, and Figure~\ref{fig:BWGD-grad-norm} shows empirical convergence to an interior stationary point.

\begin{figure}[ht!]
\begin{subfigure}{.312\linewidth}
    \centering
    \includegraphics[width=1\linewidth]{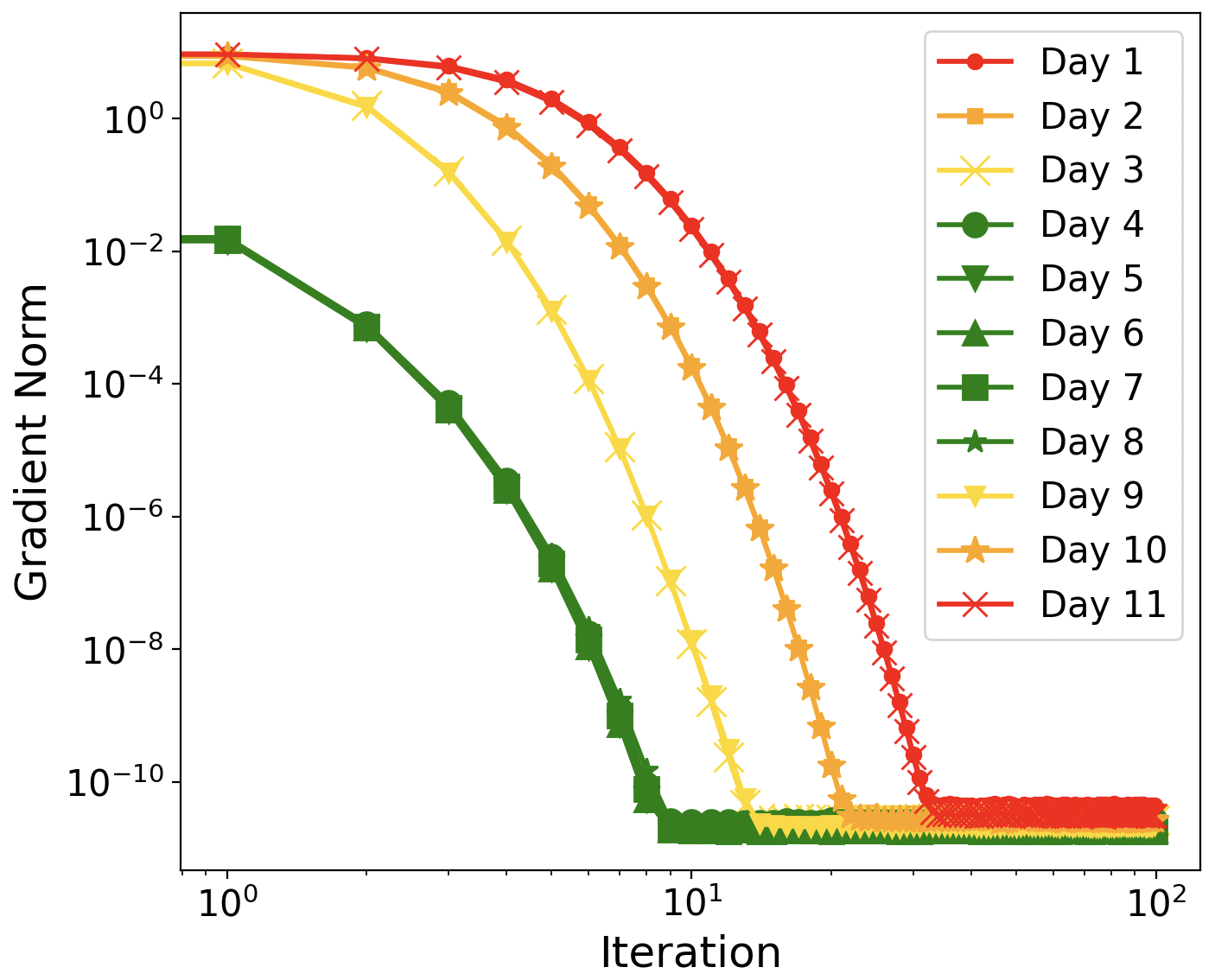}
    \caption{Gradient norms of Algorithm \ref{alg:bwgd}.}
    \label{fig:BWGD-grad-norm}
\end{subfigure}
\begin{subfigure}{.325\linewidth}
  \centering
  \includegraphics[width=1\linewidth]{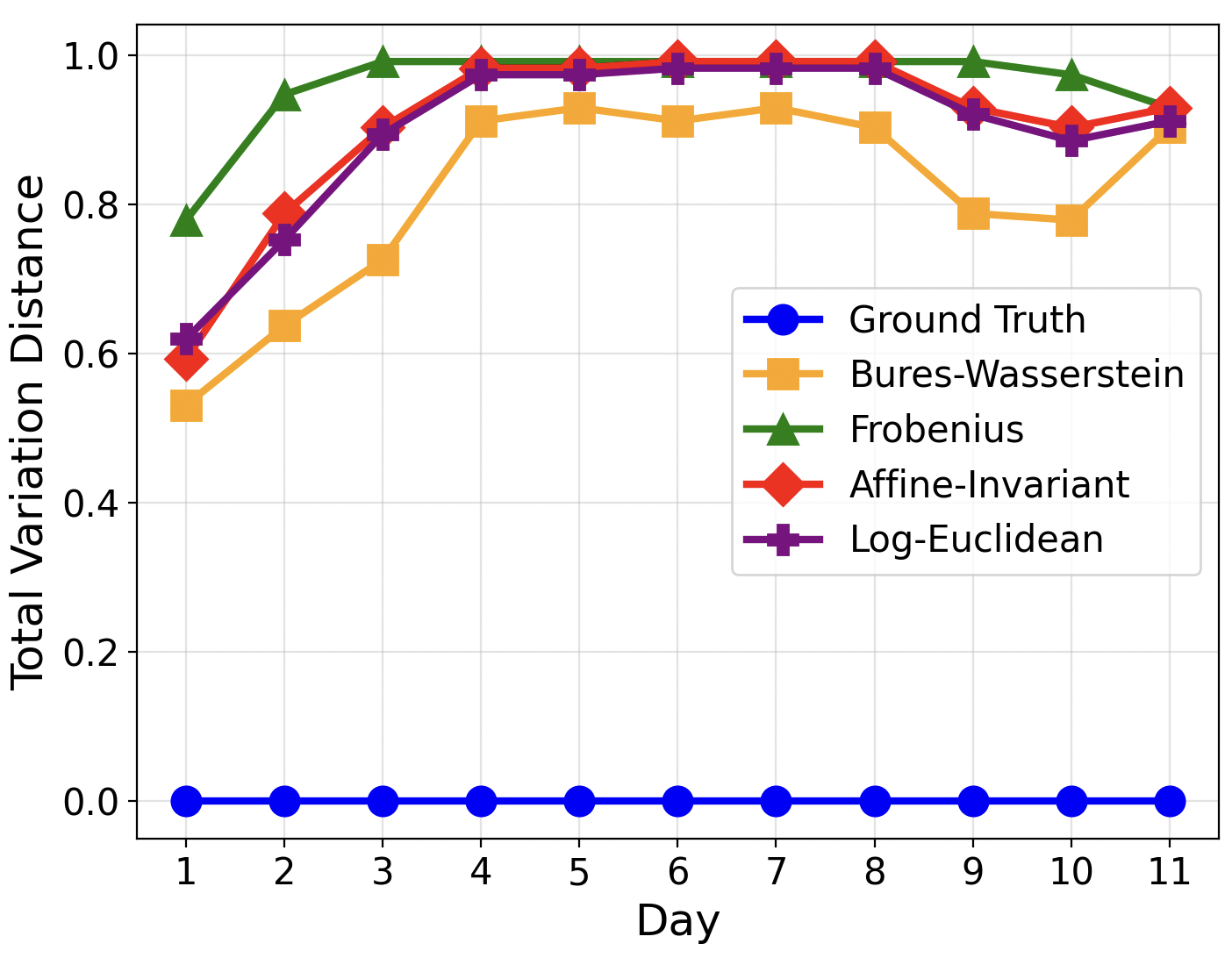}
  \caption{Degree distributions}
  \label{fig:degree_dist}
\end{subfigure}
\begin{subfigure}{.325\linewidth}
  \centering
  \includegraphics[width=1\linewidth]{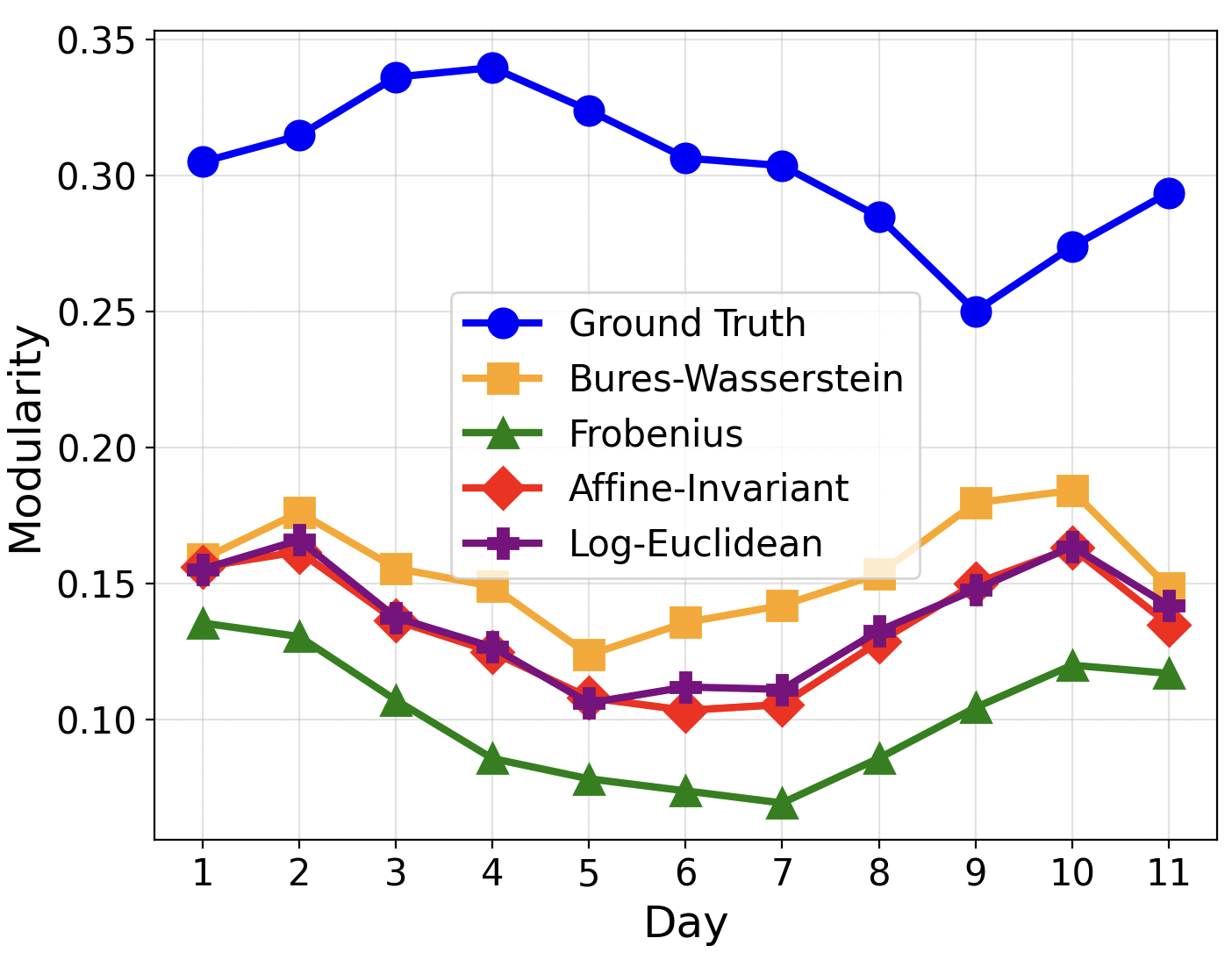}
  \caption{Modularity}
  \label{fig:modularity}
\end{subfigure}
    \caption{Results on the Ant social organization network dataset}
    \label{fig:ant_results}
\end{figure}

Moreover, to benchmark our network regression method, we compare the performance against regression approaches based on the Frobenius metric \cite{zhou2022network}, the Affine-Invariant metric \cite{lim2012matrix}, and the Log-Euclidean metric \cite{arsigny2006log}. We analyze the topological accuracy of the predictions by measuring the Total Variation (TV) distance between degree distributions and comparing modularity scores against the ground truth. Graphs are recovered as $\hat L = \hat\Sigma^{-1} - ({1}/{d})\mathbf{1}\mathbf{1}^\top$ and $\hat A = \mathrm{diag}(\mathrm{diag}(\hat L)) - \hat L$, with negative entries of $\hat A$ set to zero (no other thresholding and same rule for all metrics). Degrees are unweighted, and modularity, i.e., the excess of within-community edge weight over a degree-preserving null model~\cite{newman2006modularity}, uses edge weights. In Figure \ref{fig:degree_dist}, we plot the TV distance between the degree distributions of the predicted networks and the ground truth. Our method (BW) consistently achieves the lowest distance across all time steps, indicating that it reconstructs the node degree distribution more accurately than all other metric baselines. Furthermore, Figure \ref{fig:modularity} illustrates the preservation of community structure via the modularity score. The ground truth networks exhibit high modularity (approximately $0.30$), reflecting the strong social structure within the ant colony. While all regression methods tend to underestimate the modularity, our BW method consistently produces modularity values closer to the ground truth than the baselines (ranging from $0.13$ to $0.18$). This is a relative improvement rather than accurate recovery. The shared underestimation may stem from the common global (linear-type) Fr\'echet model, under which BW preserves the most community structure.
\subsection{Simulated Diffusion Tensors Imaging (DTI) Scalability Experiment} \label{subsec:DTI}

\begin{figure*}[ht!]
    \centering
\begin{subfigure}{.24\linewidth}
  \centering
  \includegraphics[width=\linewidth]{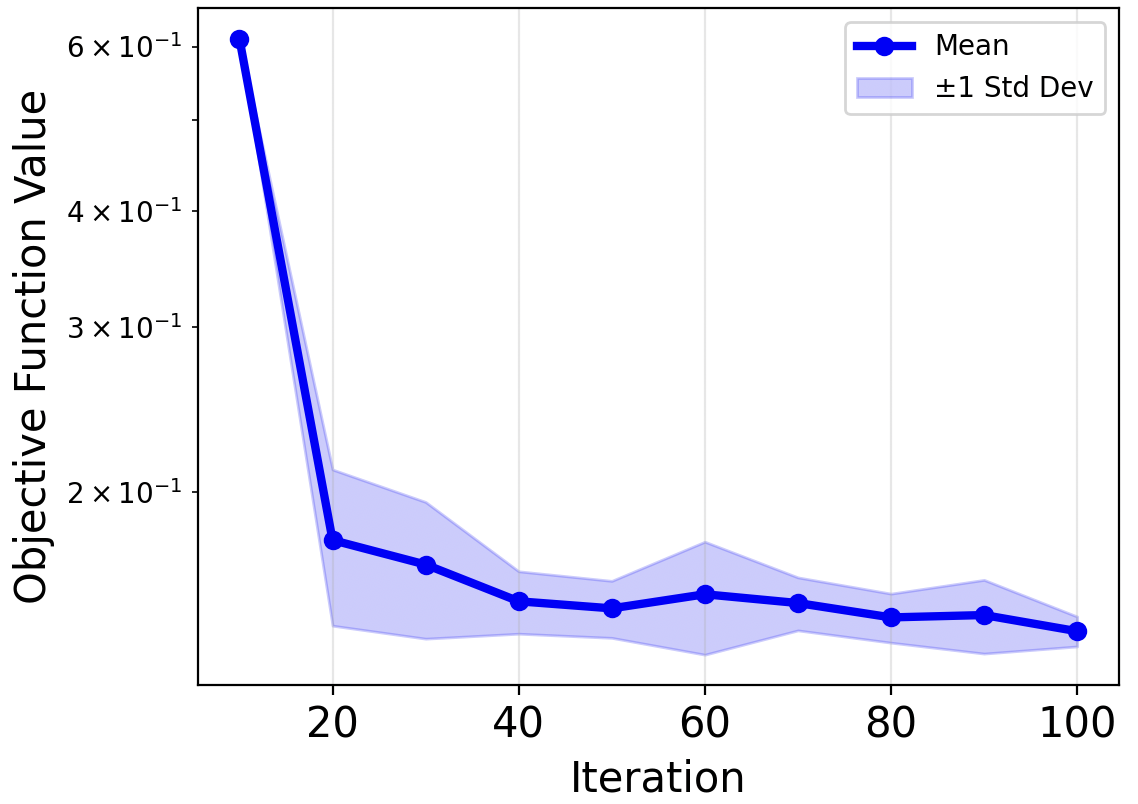}
   \vspace{-0.4cm}
  \caption{Tensor at index 20000}
  \vspace{-0.2cm}
  \label{fig:obj_20000}
\end{subfigure}
\begin{subfigure}{.24\linewidth}
  \centering
  \includegraphics[width=\linewidth]{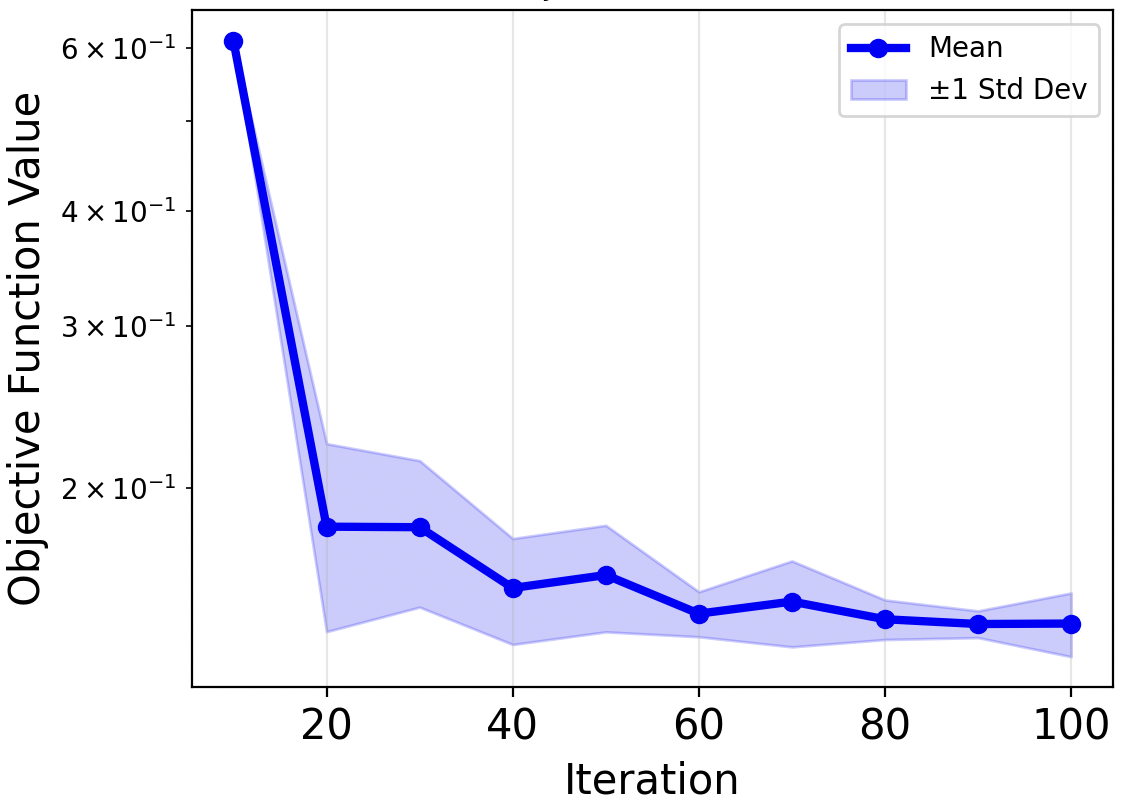}
   \vspace{-0.4cm}
  \caption{Tensor at index 40000}
   \vspace{-0.2cm}
  \label{fig:obj_40000}
\end{subfigure}
\begin{subfigure}{.24\linewidth}
  \centering
  \includegraphics[width=\linewidth]{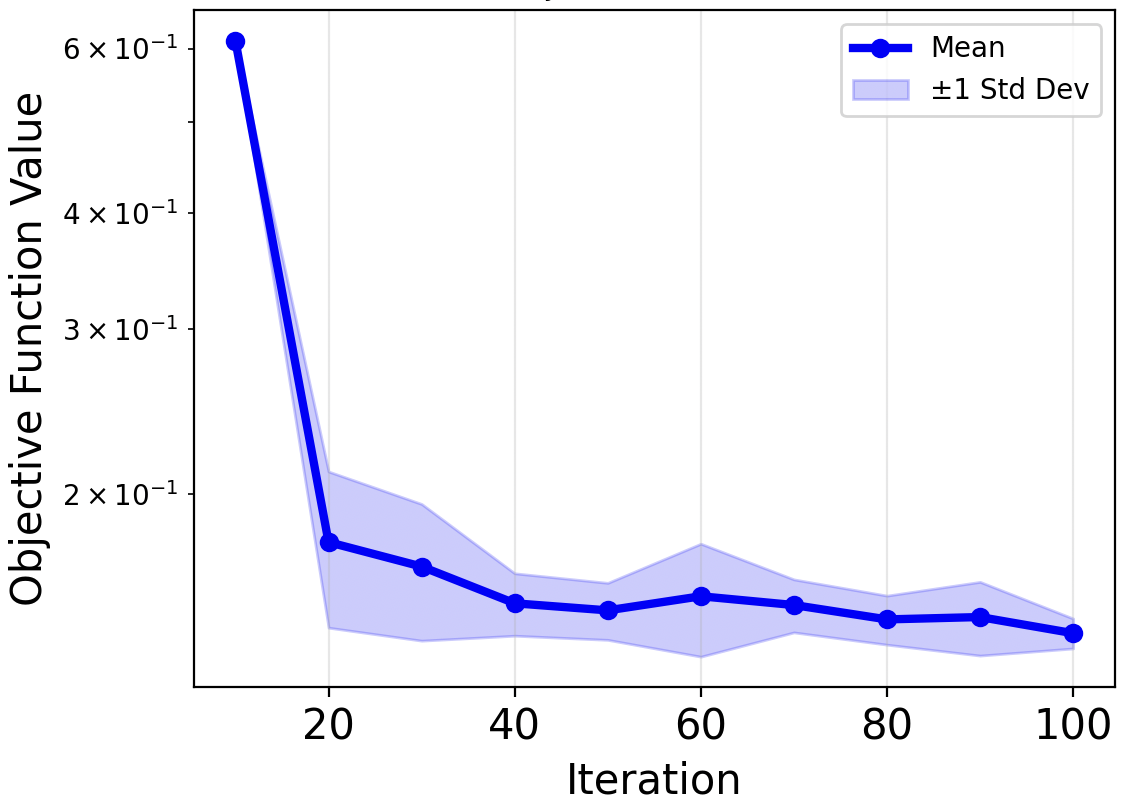}
   \vspace{-0.4cm}
  \caption{Tensor at index 60000}
   \vspace{-0.2cm}
  \label{fig:obj_60000}
\end{subfigure}
\begin{subfigure}{.24\linewidth}
  \centering
  \includegraphics[width=\linewidth]{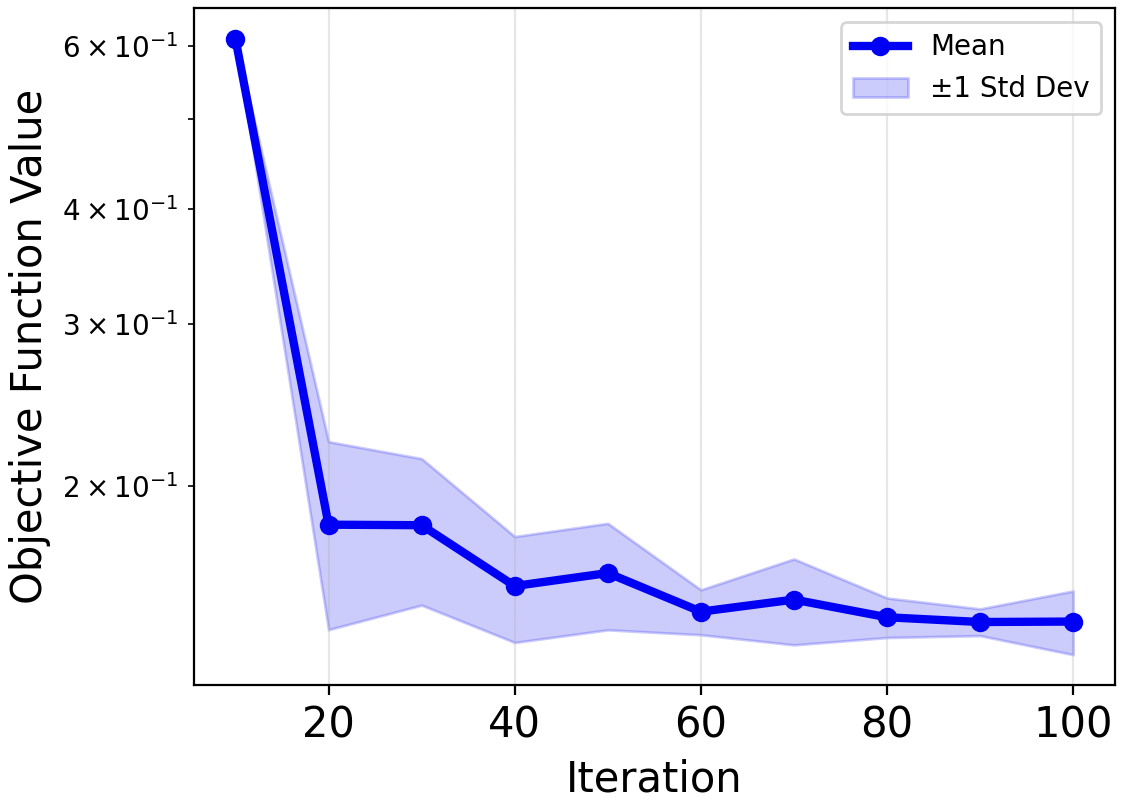}
   \vspace{-0.4cm}
  \caption{Tensor at index 80000}
   \vspace{-0.2cm}
  \label{fig:obj_80000}
\end{subfigure}
    \caption{Objective Values over 100 iterations of Algorithm \ref{alg:rsgd} from regression process of four tensors}
    \label{fig:obj_value}
\end{figure*}

This subsection is a controlled convergence and scalability study of the Pairwise Riemannian SGD (Algorithm \ref{alg:rsgd}) using a large-scale simulated dataset of $n = 100,000$ diffusion tensors. Each tensor is modeled as a $3 \times 3$ symmetric positive-definite (SPD) matrix, often visualized as an ellipsoid whose principal axes correspond to the eigenvectors and whose axial lengths correspond to the eigenvalues. This representation is standard in neuroimaging to describe the local anisotropy of water diffusion~\cite{pennec2006riemannian}. The simulation does not model acquisition, Rician noise, or tensor estimation. Realistic DT-MRI validation is beyond our scope. The synthetic tensors are generated along a helical backbone defined by the spatial coordinates $x(t) = 10\cos(t)$, $y(t) = 10\sin(t)$, and $z(t) = 5t$ for $t \in [0, 2\pi]$, with principal orientations aligned to the helix tangent. We provide a visualization for 20 tensors along the helix in Figure~\ref{fig:helix-visualization}. We also checked that our Spectral Dominance condition \eqref{eq:SNM} holds for all 20 tensors, although negative weights are present at every single evaluation point. More details about this evaluation are presented in Appendix~\ref{appen:condition-validation}. In this experiment, we select four target indices: 20,000, 40,000, 60,000, and 80,000 (equally spaced across the 100,000-point trajectory), regress each of these four points using the Fréchet Regression model based on all other points, and compare the results to the ground-truth points.
To solve the regression problem, we use Algorithm \ref{alg:rsgd}, with a diminishing learning rate $\eta_t = \eta_0 / \sqrt{t+1}$ and $\eta_0 = 1$. The experimental protocol involves 10 independent runs per target point, each lasting 100 iterations. The objective value is recorded every 10 iterations to monitor stochastic convergence. In Appendix~\ref{subsec:appen-DTI}, we also provide visualizations for ground-truth and predicted tensor shapes following through the helix. 

As shown in Figure \ref{fig:obj_value}, the convergence results, visualized by the mean and $\pm 1$ standard deviation of the objective values, highlight the algorithm's efficiency and stability. Across all four target points, there is a sharp initial decline, with the objective value dropping from approximately $0.6$ to below $0.2$ in just 20 iterations. While the intermediate phase shows some stochastic sampling variance, the diminishing step size eventually dampens these oscillations. In the final iterations, the variance significantly narrows, and the objective values reach a consistent plateau. This terminal stability empirically confirms that the algorithm is robust to initialization and that the iterates naturally remain within the SPD cone $\mathbb{S}_{++}^3$ without requiring expensive projection steps. 
In terms of scalability, the direct baseline is Algorithm~\ref{alg:bwgd}, which optimizes the same objective: full-batch Algorithm~\ref{alg:bwgd} requires $\approx$4s per iteration ($\approx$400s total), while Pairwise Algorithm~\ref{alg:rsgd} completes 100 iterations in $\approx$7s (57$\times$ speedup), with per-iteration cost independent of $n$ (Appendix~\ref{app:wallclock}). Since each iteration still costs $O(d^3)$, this scalability is in $n$, not in $d$.
\section{Conclusion} \label{sec:conclusion}
This paper establishes a comprehensive framework for Fréchet regression on the Bures-Wasserstein (BW) manifold, specifically addressing extrapolation with signed weights. We introduced the Spectral Dominance condition, which is sufficient to guarantee the existence of a minimizer, and provided a stricter condition for uniqueness. Our landscape analysis further proves that the objective is free of local maxima. Building on these foundations, we contributed two projection-free algorithms: a full-batch Riemannian Gradient Descent with a proven sublinear convergence rate and a novel pairwise Stochastic Gradient Descent (R-SGD) reformulation. Empirically, BW regression outperformed Frobenius, Affine-Invariant, and Log-Euclidean baselines in preserving topological features (degree distributions and modularity) in biological networks and demonstrated remarkable scalability on $100,000$ tensors in DTI simulations.

On the other hand, the Spectral Dominance condition is sufficient but not necessary for the existence of a minimizer, and characterizing the tight feasibility boundary remains an open problem. Moreover, our conditions open the door to more sophisticated optimization strategies, including accelerated, distributed, or adaptive Riemannian methods to achieve faster convergence. Another direction is to extend this framework to the closure of the SPD cone to handle positive-semidefinite matrices, which frequently arise in low-rank manifold learning. Finally, while our Pairwise R-SGD handles large sample sizes $n$, future research will focus on scaling to very large dimensions $d$, making BW regression a viable tool for high-dimensional covariance models.

\section*{Acknowledgement}

This work was supported in part by the National Science Foundation under Grant \#2443064 and in part by the Air Force under MOU FA8750-19-3-1000. The U.S. Government is authorized to reproduce and distribute copies for Governmental purposes notwithstanding any copyright or other restrictive legends. The views and conclusions contained herein are those of the authors and should not be interpreted as necessarily representing the official policies or endorsements, either expressed or implied, of the Air Force or the U.S. Government.

\newpage
\bibliography{neurips_2026}
\bibliographystyle{plain}

%%%%%%%%%%%%%%%%%%%%%%%%%%%%%%%%%%%%%%%%%%%%%%%%%%%%%%%%%%%%
\newpage
\appendix
\section*{Appendix}
In this appendix, we first introduce the preliminaries on the Bures-Wasserstein manifold and notations, and derive Algorithm~\ref{alg:bwgd} as Riemannian gradient descent (Section~\ref{appen:prelim}). Second, we summarize related regression methods and compare the main geometries on the SPD cone (Section~\ref{appen:metric-comparison}). Third, we present an example in which the Bures-Wasserstein prediction remains finite while the Affine-Invariant and Log-Euclidean predictions diverge in graph-Laplacian extrapolation (Section~\ref{appen:BW-AI-LE}). Fourth, we present detailed visualizations of two experiments from the main paper (Section~\ref{appen:visualization}). Fifth, we show the proofs of all theorems and propositions precisely in Section~\ref{appen:proof}. Sixth, we indicate the validation of our Spectral Dominance condition in two numerical experiments (Section~\ref{appen:condition-validation}). Lastly, we present the wall-clock time for the DTI experiment (Section~\ref{app:wallclock}).

\section{Preliminaries on the Bures-Wasserstein Geometry} \label{appen:prelim}

The Bures-Wasserstein manifold $\mathbb{S}_{++}^d$ is the space of $d \times d$ symmetric positive definite (SPD) matrices equipped with the 2-Wasserstein metric arising from optimal transport between zero-mean Gaussian measures. This manifold has a rich geometric structure that enables smooth optimization of covariance matrices and kernel operators. For $S_1, S_2 \in \mathbb{S}_{++}^d$, the squared Bures-Wasserstein distance is
\begin{equation*}
    W_2^2(S_1, S_2)
    = \mathrm{Tr}(S_1) + \mathrm{Tr}(S_2)
    - 2\,\mathrm{Tr}\!\left( (S_1^{1/2} S_2 S_1^{1/2})^{1/2} \right).
\end{equation*}
This is precisely the $2$-Wasserstein distance between the zero-mean Gaussians
$\mathcal{N}(0,S_1)$ and $\mathcal{N}(0,S_2)$. The tangent space at $X\in\mathbb{S}_{++}^d$, denoted by $T_{X}\mathbb{S}_{++}^d$, is the linear space
$\mathrm{Sym}(d)$ of symmetric matrices. The Bures-Wasserstein Riemannian metric
is defined using the Lyapunov operator, $\mathcal{L}_X[U]$, which is the solution of $X Z + Z X = U$; this solution is unique since $X \succ 0$. For $U,V\in T_X\mathbb{S}_{++}^d$, the Riemannian metric is $\langle U,V\rangle_X := \frac{1}{2}\,\mathrm{Tr}\!\left( \mathcal{L}_X[U]\, V \right)$, with norm $\lVert U\rVert_X^2 := \langle U,U\rangle_X$. The Riemannian exponential map is $\mathrm{Exp}_X(U) = X + U + \mathcal{L}_X[U]\, X\, \mathcal{L}_X[U] = (I + \mathcal{L}_X[U])\,X\,(I + \mathcal{L}_X[U])$, for $U \in T_X \mathbb{S}_{++}^d$ with $I + \mathcal{L}_X[U] \succ 0$. Given a smooth function $f:\mathbb{S}_{++}^d \to \mathbb{R}$, with Euclidean gradient
$\nabla f(X)$, the Bures-Wasserstein gradient is $\nabla_{\mathrm{bw}} f(X)
    = 4\,\{ \nabla f(X)\, X \}_{\mathrm{s}},$ where $\{\cdot\}_{\mathrm{s}}$ denotes the symmetrization operator
$\{A\}_{\mathrm{s}} = \tfrac{1}{2}(A + A^\top)$. The manifold $\mathbb{S}_{++}^d$ is geodesically convex and non-negatively curved \protect{\citep[Theorem 1.1]{takatsu2010wasserstein}}. Between any
$S_1, S_2 \in \mathbb{S}_{++}^d$, the unique constant-speed geodesic is $\gamma(t)
= \big( (1-t)I + t\,T \big)\, S_1 \,\big( (1-t)I + t\,T \big),
\text{ for }t\in[0,1],$ where $T = S_1^{-1/2} (S_1^{1/2} S_2 S_1^{1/2})^{1/2} S_1^{-1/2}.$ The term $T$ is also equivalent to the geometric mean of $S_1^{-1}$ and $S_2$. We denote the geometric mean between two SPD matrices $A,B$ as $\mathrm{GM}(A,B) = A^{1/2} (A^{-1/2} B A^{-1/2})^{1/2} A^{1/2}$. The matrix $T$ is SPD and satisfies $T S_1 T = S_2$, i.e., $T$ is the optimal transport map from $\mathcal{N}(0,S_1)$ to $\mathcal{N}(0,S_2)$. Since $\gamma(t) = \mathrm{Exp}_{S_1}(t\gamma'(0))$, differentiating $\gamma$ at $t=0$ gives~\cite{han2021riemannian,thanwerdas2023bures}
$$\log_{S_1}(S_2) = \gamma'(0) = (T - I)S_1 + S_1(T - I).$$
Finally, for $\Sigma \in \mathbb{S}_{++}^d$, the Euclidean gradient of $W_2^2(\cdot,\Sigma)$ at $X$ is
$$\nabla_X W_2^2(X,\Sigma) = I - \mathrm{GM}(X^{-1},\Sigma).$$
Indeed, the matrices $X^{1/2}\Sigma X^{1/2}$ and $\Sigma^{1/2}X\Sigma^{1/2}$ have the same eigenvalues, so
$$\mathrm{Tr}\big((X^{1/2}\Sigma X^{1/2})^{1/2}\big) = \mathrm{Tr}\big((\Sigma^{1/2}X\Sigma^{1/2})^{1/2}\big).$$
Since the derivative of $M \mapsto \mathrm{Tr}(M^{1/2})$ is $\frac{1}{2}M^{-1/2}$, the Euclidean gradient of the right-hand side with respect to $X$ is
$$\tfrac{1}{2}\Sigma^{1/2}\big(\Sigma^{1/2}X\Sigma^{1/2}\big)^{-1/2}\Sigma^{1/2} = \tfrac{1}{2}\mathrm{GM}(\Sigma, X^{-1}) = \tfrac{1}{2}\mathrm{GM}(X^{-1},\Sigma),$$
where the last equality is the symmetry of the geometric mean. For each matrix $S \in \mathbb{S}_{++}^d$, we denote the minimum and maximum eigenvalues of $S$ as $\lambda_{\min}(S)$ and $\lambda_{\max}(S)$, respectively.

\subsection*{Algorithm~\ref{alg:bwgd} as Riemannian gradient descent}
We derive the update of Algorithm~\ref{alg:bwgd} from the BW geometry above, following the positive-weight analysis of~\cite{chewi2020gradient}. Consider the scaled objective
$$\Phi(S) = \frac{1}{2}\sum_{k=1}^n \lambda_k W_2^2(S,\Sigma_k) = \frac{1}{2}F(S),$$
which has the same minimizers and stationary points as Problem~\eqref{eqn:frechet-obj}. For $S \in \mathbb{S}_{++}^d$, let $T_k := \mathrm{GM}(S^{-1},\Sigma_k)$, i.e., the matrix $T$ above with $(S_1,S_2) = (S,\Sigma_k)$. Then, $\log_S(\Sigma_k) = (T_k - I)S + S(T_k - I)$.

\textbf{Riemannian gradient.} By the gradient formula above and $\sum_k \lambda_k = 1$, the Euclidean gradient of $\Phi$ at $S$ is $\frac{1}{2}\sum_k\lambda_k(I - T_k) = \frac{1}{2}(I - \sum_k \lambda_k T_k)$. Applying the formula for $\nabla_{\mathrm{bw}}$ and using $\sum_k \lambda_k = 1$,
$$-\eta\nabla_{\mathrm{bw}}\Phi(S) = \eta\Big[\Big(\sum_{k=1}^n \lambda_k T_k - I\Big)S + S\Big(\sum_{k=1}^n \lambda_k T_k - I\Big)\Big] = \eta \sum_{k=1}^n \lambda_k \log_S(\Sigma_k).$$

\textbf{Update.} The display above shows that $\eta(\sum_k \lambda_k T_k - I)$ solves the Lyapunov equation with right-hand side $-\eta\nabla_{\mathrm{bw}}\Phi(S)$. By uniqueness of this solution,
$$I + \mathcal{L}_S\big[-\eta\nabla_{\mathrm{bw}}\Phi(S)\big] = (1-\eta)I + \eta\sum_{k=1}^n \lambda_k T_k = \tilde S.$$
Whenever $\tilde S \succ 0$, which holds for $\eta \in (0,1]$ under~\eqref{eq:SNM} by Proposition~\ref{prop:non-projection-iteration}, the exponential map gives $\mathrm{Exp}_S(-\eta\nabla_{\mathrm{bw}}\Phi(S)) = \tilde S\, S\, \tilde S$. These are exactly Lines~3--4 of Algorithm~\ref{alg:bwgd}. Thus, $\tilde S_t$ is the congruence factor of the exponential map, not the tangent vector $-\eta\nabla_{\mathrm{bw}}\Phi(S_{t-1})$, and no Lyapunov equation has to be solved numerically. For non-negative weights and $\eta = 1$, the update reduces to the fixed-point iteration $S \mapsto (\sum_k \lambda_k T_k)\,S\,(\sum_k \lambda_k T_k)$ of~\cite{alvarez2016fixed,chewi2020gradient}. Replacing $\sum_k \lambda_k T_k$ with $\mu_+T_i - \mu_-T_j$ (note that $\mu_+ - \mu_- = 1$) gives Lines~4--5 of Algorithm~\ref{alg:rsgd} as the exponential step along $-\eta_t\nabla_{\mathrm{bw}}\phi_{ij}$ with $\phi_{ij} = \frac{1}{2}f_{ij}$.

\section{Related regression methods and geometries on the SPD cone}\label{appen:metric-comparison}
\subsection{Regression methods for manifold- and metric-space-valued responses}
Table~\ref{tab:regression-methods} summarizes the regression methods for manifold- and metric-space-valued responses discussed in the Introduction and Section~\ref{sec:related-work}. The \emph{Signed weights} column indicates how each method handles predictions with negative (extrapolating) weights.
\vspace{-0.3cm}
\begin{table}[h!]
\centering
\caption{Regression methods for manifold- and metric-space-valued responses (wFM: weighted Fr\'echet mean; F: Frobenius).}
\label{tab:regression-methods}
\vspace{0.1cm}
\footnotesize
\setlength{\tabcolsep}{4pt}
\begin{tabular}{@{}lllll@{}}
    \toprule
    Method & $X$ & $Y$ (geometry) & Regressor & Signed weights \\
    \midrule
    Davis~\cite{davis2010population} & $\mathbb{R}$ & manifold & kernel wFM & n/a ($w\geq 0$) \\
    Fletcher~\cite{thomas2013geodesic} & $\mathbb{R}$ & manifold & geodesic & n/a (geodesic) \\
    Zhu; Yuan~\cite{zhu2009intrinsic,yuan2012local} & Euclidean & SPD (AI, LE) & link; loc. poly. & -- \\
    Banerjee~\cite{banerjee2015nonlinear,banerjee2016nonlinear} & Euclidean/manifold & manifold & kernel & -- \\
    Petersen~\cite{petersen2019frechet} & Euclidean & metric space & glob./loc. wFM & assumed \\
    Zhou~\cite{zhou2022network} & Euclidean & Laplacians (F) & glob./loc. wFM & projection \\
    Fan~\cite{fan2024conditional} & Euclidean & distrib. ($W_2$) & cond. barycenter & single geodesic \\
    Xu~\cite{xu2025wasserstein} & Euclidean & SPD (BW) & global wFM & assumed \\
    LE~\cite{arsigny2007geometric,bouza2021higher} & Euclidean & SPD (LE) & $\exp(\sum_k\lambda_k\log\Sigma_k)$ & closed form \\
    \midrule
    \textbf{Ours} & Euclidean & SPD (BW) & global wFM & \textbf{certified,} \\
    & & & & \textbf{projection-free} \\
    \bottomrule
    \end{tabular}
    \vspace{-0.6cm}
\end{table}

\subsection{Geometries on the SPD cone}

\begin{table}[!h]
\centering
\scriptsize
\setlength{\tabcolsep}{2.5pt}
\renewcommand{\arraystretch}{1.3}
\newcolumntype{R}[1]{>{\raggedright\arraybackslash}p{#1}}
\begin{tabular}{@{}R{1.9cm}R{4.25cm}R{1.45cm}R{2.2cm}R{0.95cm}R{2.3cm}@{}}
\toprule
Geometry & $d^2(A,B)$ & Curvature; completeness & Invariances & Main cost & Signed wFM \\
\toprule
Frobenius (F) & $\|A-B\|_F^2$ & flat; incomplete & orthogonal congruence & $O(d^2)$ & $\sum_k\lambda_k\Sigma_k$; may leave $\mathbb{S}_{++}^d$ \\
\hline
Affine-Invariant (AI)~\cite{pennec2006riemannian,moakher2005differential} & $\|\log(A^{-1/2}BA^{-1/2})\|_F^2$ & $\leq 0$; complete & congruence by $GL_d$ (hence scaling); inversion & eig., $O(d^3)$ & exists (by completeness); iterative \\
\hline
Log-Euclidean (LE)~\cite{arsigny2006log,arsigny2007geometric} & $\|\log A - \log B\|_F^2$ & flat; complete & orthogonal congruence; scaling; inversion & eig., $O(d^3)$ & $\exp(\sum_k\lambda_k\log\Sigma_k)$ \\
\hline
Power-Euclidean~\cite{dryden2009non} & $\alpha^{-2}\|A^\alpha - B^\alpha\|_F^2$, $\alpha>0$ & flat; incomplete & orthogonal congruence & eig., $O(d^3)$ & $(\sum_k\lambda_k\Sigma_k^\alpha)^{1/\alpha}$ if $\sum_k\lambda_k\Sigma_k^\alpha\succ 0$ \\
\hline
Log-Cholesky~\cite{lin2019riemannian}; Cholesky-based~\cite{chen2026fast} & $\|\lfloor L_A\rfloor - \lfloor L_B\rfloor\|_F^2 + \|\log \mathbb{D}(L_A) - \log\mathbb{D}(L_B)\|_F^2$, where $A = L_AL_A^\top$ & flat; complete & not orthogonally invariant & Cholesky, $O(d^3)$ & closed form in Cholesky coordinates \\
\hline
Stein~\cite{sra2012new,salehian2013recursive} & $\log\det\frac{A+B}{2} - \frac{1}{2}\log\det(AB)$ & divergence-based metric & congruence by $GL_d$; inversion & log-det, $O(d^3)$ & -- \\
\hline
Scaling-rotation~\cite{jung2015scaling} & defined via eigen-decompositions (scaling and rotation) & quotient construction & orthogonal congruence & eig., $O(d^3)$ & -- \\
\hline
Bures-Wasserstein (BW)~\cite{bhatia2019bures,takatsu2010wasserstein} & $\mathrm{Tr}A + \mathrm{Tr}B - 2\,\mathrm{Tr}(A^{1/2}BA^{1/2})^{1/2}$ & $\geq 0$; incomplete & orthogonal congruence; $W_2(cA,cB) = \sqrt{c}\,W_2(A,B)$ & sqrt, $O(d^3)$ & exists under~\eqref{eq:SNM} (Thm.~\ref{thm:existence}); projection-free \\
\hline
Generalized BW~\cite{han2021generalized} & BW geometry parameterized by $M \in \mathbb{S}_{++}^d$ ($M = I$ gives BW) & as BW & depends on $M$ & sqrt, $O(d^3)$ & -- \\
\bottomrule
\end{tabular}
\vspace{0.2cm}
\caption{Main geometries on $\mathbb{S}_{++}^d$. $\lfloor L\rfloor$ and $\mathbb{D}(L)$ denote the strictly lower-triangular and diagonal parts of a Cholesky factor $L$; properties in the Log-Cholesky row refer to~\cite{lin2019riemannian}. For Stein, the listed quantity is the S-divergence, whose square root is a metric~\cite{sra2012new}. ``--'': not covered by a general result that we are aware of.}
\label{tab:metric-comparison}
\vspace{-0.8cm}
\end{table}

Table~\ref{tab:metric-comparison} compares the main geometries on $\mathbb{S}_{++}^d$ in terms of curvature, completeness, invariance, cost, and the behavior of the signed weighted Fr\'echet mean (wFM) $\arg\min_S \sum_k \lambda_k d^2(S,\Sigma_k)$ with $\sum_k\lambda_k = 1$. The comparison is meant to position BW, not to rank the metrics: the appropriate geometry depends on the application.

\textbf{Swelling and shape.} Frobenius interpolation of SPD matrices inflates determinants (the swelling effect), whereas AI and LE interpolation avoid it~\cite{arsigny2007geometric}. BW only mitigates it (Proposition~\ref{prop:swelling}). BW also does not preserve tensor shape in general (Figure~\ref{fig:ellipsoid-geodesic}). When orientation and shape must be treated separately, scaling-rotation~\cite{jung2015scaling} or group-action-induced constructions~\cite{cetingul2012group} may be preferable.

\textbf{Tangent-space constructions.} Bouza et al.~\cite{bouza2021higher} compute a local Fr\'echet mean $\bar\Sigma$, map observations to the tangent space at $\bar\Sigma$ with $\log_{\bar\Sigma}$, take an unrestricted linear combination there, and map back with $\exp_{\bar\Sigma}$. Applied globally with the identity as anchor under the LE geometry, this gives $\hat S_{\mathrm{LE}} = \exp(\sum_k \lambda_k \log\Sigma_k)$. Since $\sum_k\lambda_k = 1$, $\sum_k\lambda_k\|\log S - \log\Sigma_k\|_F^2 = \|\log S - \sum_k\lambda_k\log\Sigma_k\|_F^2 + C$, so $\hat S_{\mathrm{LE}}$ is the exact and unique minimizer of the signed LE objective for arbitrary real weights; it is the LE baseline of Subsection~\ref{subsec:ant-experiment}. For a general manifold, an anchored log-exp construction (or its retraction-based approximation) is globally defined when the injectivity radius is infinite, but it does not in general minimize the intrinsic signed squared-distance objective; the LE case is special because the logarithm identifies $\mathbb{S}_{++}^d$ with a flat space. Such constructions, therefore, change the response geometry and, hence, the regression target, whereas we study the intrinsic signed BW objective.

\textbf{Signed weights and completeness.} In a complete geometry (AI, LE, Log-Cholesky), the triangle inequality gives $d^2(S,\Sigma_k) = d^2(S,\Sigma_0) + O(d(S,\Sigma_0))$, so, since $\sum_k\lambda_k = 1$, the signed objective is coercive and admits a minimizer for arbitrary real weights; LE and Log-Cholesky even give closed forms. Under BW, the cone is incomplete, and minimizing sequences can approach singular matrices (Example~\ref{example}), which is why Condition~\eqref{eq:SNM} is needed. These geometries, however, define different regression targets.

\vspace{0.2cm}
\subsection{Advantages of the Bures-Wasserstein geometry}\label{appen:BW-advantages}
The geometries in Table~\ref{tab:metric-comparison} are well developed, and BW is not uniformly preferable. We collect the properties that make BW well-suited to the \emph{signed} conditional barycenter problem in particular. Each is standard or proved below.

\textbf{Transport and Procrustes interpretation.} $W_2(A,B)$ is the $2$-Wasserstein distance between $\mathcal{N}(0,A)$ and $\mathcal{N}(0,B)$~\cite{takatsu2010wasserstein,bhatia2019bures} and equals the Procrustes size-and-shape distance $\min_{R \in O(d)}\|A^{1/2} - B^{1/2}R\|_F$~\cite{dryden2009non}. Hence, Problem~\eqref{eqn:frechet-obj} is a conditional Wasserstein barycenter of Gaussians, and for the graph representation of Subsection~\ref{subsec:ant-experiment} it compares distributions of smooth graph signals rather than Laplacian entries~\cite{haasler2024bures}. Minimizing geodesics are push-forwards by linear optimal transport maps, which gives the closed-form exponential and logarithmic maps of Appendix~\ref{appen:prelim}.

\textbf{Continuity up to the boundary.} The expression for $W_2^2$ is finite and continuous on the closed cone $\mathbb{S}_+^d$, whereas the AI and LE distances from any SPD matrix to a singular one are infinite. First, $F$ extends continuously to $\mathbb{S}_+^d$, which Step~1 of the proof of Theorem~\ref{thm:existence} uses: we minimize on the closed cone and then exclude the boundary under~\eqref{eq:SNM}. Second, signed BW predictions can remain bounded for near-singular responses such as regularized Laplacians $L + \varepsilon I$, while AI and LE predictions diverge as $\varepsilon \to 0$ (Appendix~\ref{appen:BW-AI-LE}).

\textbf{Between Frobenius and logarithmic interpolation.} The following makes ``BW mitigates swelling'' precise.
\begin{Proposition}\label{prop:swelling}
Let $A,B \in \mathbb{S}_{++}^d$, $T := \mathrm{GM}(A^{-1},B)$, and $\gamma(t) := ((1-t)I + tT)A((1-t)I + tT)$ be the BW geodesic, for $t \in [0,1]$. Then,
\begin{align*}
(1-t)A + tB - \gamma(t) &= t(1-t)(T-I)A(T-I) \succeq 0, \\
\det(A)^{1-t}\det(B)^{t} &\leq \det\gamma(t) \leq \det\big((1-t)A + tB\big).
\end{align*}
\end{Proposition}
\begin{proof}
Expanding $\gamma(t)$ and using $TAT = B$ gives the identity, and the upper determinant bound follows from the Loewner order. For the lower bound, $\det\gamma(t) = \det((1-t)I + tT)^2\det A$ and $\det T = (\det B/\det A)^{1/2}$. Minkowski's determinant inequality gives $\det((1-t)I+tT)^{1/d} \geq (1-t) + t\det(T)^{1/d}$, so $\det\gamma(t)^{1/(2d)} \geq (1-t)\det(A)^{1/(2d)} + t\det(B)^{1/(2d)}$, and the weighted AM-GM inequality concludes.
\end{proof}
Since the AI and LE geodesics have determinant $\det(A)^{1-t}\det(B)^t$, BW interpolation never exceeds Frobenius interpolation in the Loewner order and never has a smaller determinant than AI or LE interpolation.

\textbf{Closed-form, projection-free updates.} Each Riemannian step requires one matrix square root per observation, plus a square root and an inverse of the iterate, and no matrix logarithm, exponential, or numerical Lyapunov solve (Appendix~\ref{appen:prelim}). Each step is a congruence $S \mapsto \tilde S S \tilde S$, so feasibility reduces to the linear matrix inequality $\tilde S \succ 0$, which Proposition~\ref{prop:non-projection-iteration} certifies for every $S$ under~\eqref{eq:SNM}. Moreover, the BW metric depends linearly on the base point, compared with the quadratic dependence of the AI metric, which yields better-conditioned Riemannian Hessians for several learning objectives~\cite{han2021riemannian}.

\textbf{Trade-offs.} BW is neither affine- nor scale-invariant, and it is non-negatively curved and incomplete, so signed existence requires~\eqref{eq:SNM} and the smoothness constant depends on $\lambda_{\min}$ of the iterates. When these properties are central, the alternatives in Table~\ref{tab:metric-comparison} may be preferable.

\section{Instability of Affine-Invariant and Log-Euclidean Metrics under Signed Weights}
\label{appen:BW-AI-LE}
We provide a closed-form example showing why the Bures-Wasserstein (BW) metric is a viable choice for signed Fr\'{e}chet regression on graph Laplacians, while Affine-Invariant (AI) and Log-Euclidean (LE) metrics diverge. This divergence is not a numerical artifact but a consequence of the logarithmic singularity of the AI and LE geometries at the boundary of the cone, whereas $W_2$ extends continuously to singular matrices (Appendix~\ref{appen:BW-advantages}).
\begin{figure}[h!]
    \centering
\begin{subfigure}{.48\linewidth}
  \centering
  \includegraphics[width=\linewidth]{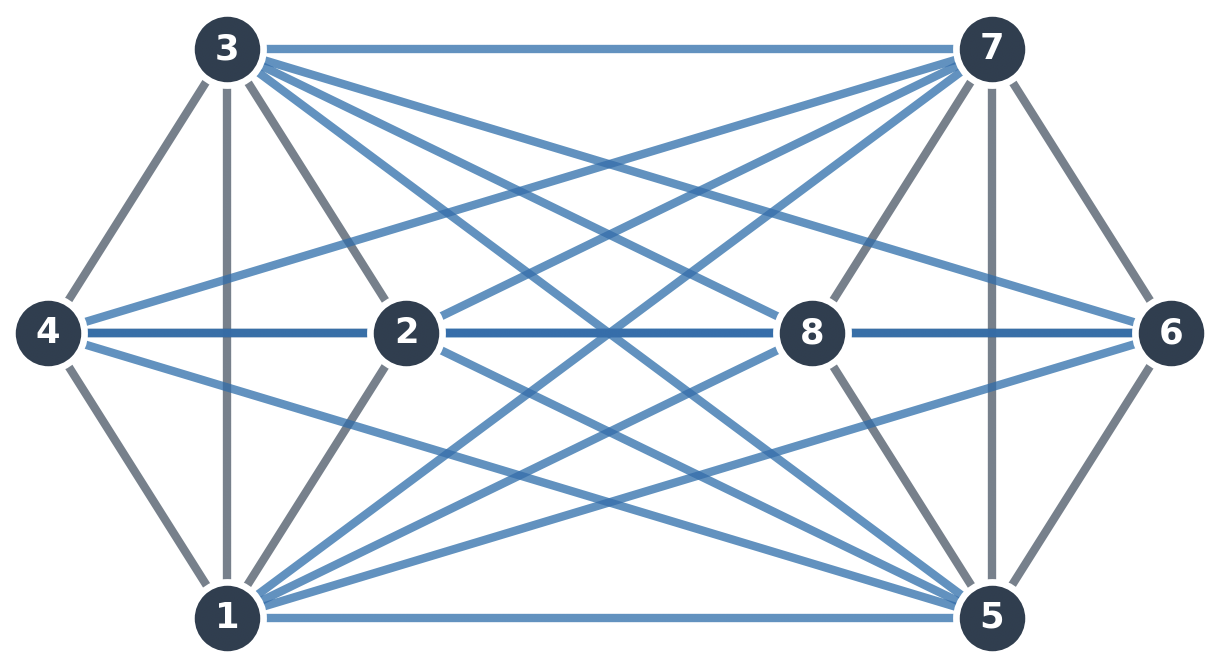}
  \caption{Graph $G_1 = K_8$}
  \label{fig:graph-G_1-visualization}
\end{subfigure}
\hspace{0.3cm}
\begin{subfigure}{.48\linewidth}
  \centering
  \includegraphics[width=\linewidth]{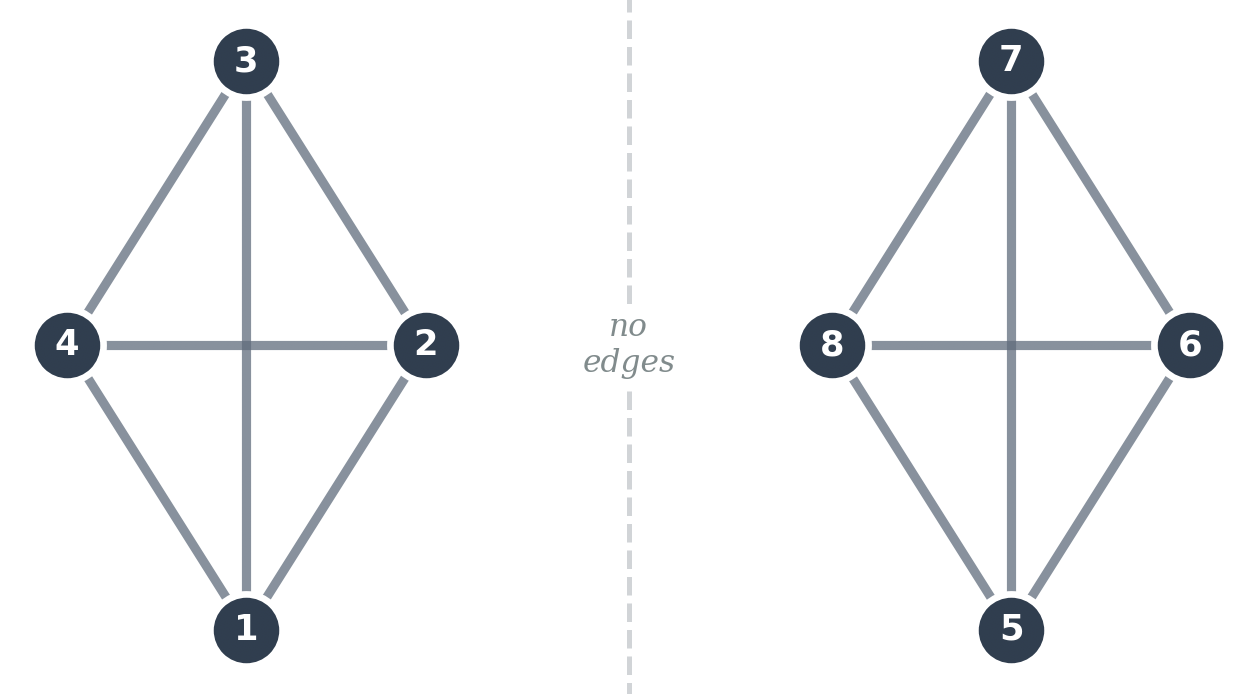}
  \caption{Graph $G_2 = K_4 \sqcup K_4$}
  \label{fig:graph-G_2-visualization}
\end{subfigure}
    \caption{Original graphs $G_1$ and $G_2$}
    \label{fig:graph_G_1_G_2}
\end{figure}
\paragraph{Setup.}
Let $G_1 = K_8$ (complete graph) and $G_2 = K_4 \sqcup K_4$
(two disjoint 4-cliques), with unit edge weights (Figure~\ref{fig:graph_G_1_G_2}). Their graph Laplacians are $L_1 = 8I - J_8$ and $L_2 = \mathrm{blkdiag}(4I{-}J_4,\,4I{-}J_4)$, where $J_n =\mathbf{1}\mathbf{1}^\top$. Since $L_1$ and $L_2$ commute (verified by $J_8 B = 4J_8$ with $B = \mathrm{blkdiag}(J_4,J_4)$), they share an eigenbasis, and the barycenter problem decouples into independent scalar problems per eigenspace. We regularize $S_k = L_k + \varepsilon I$ ($\varepsilon > 0$) to ensure positive definiteness, as required by AI and LE.
 
\paragraph{Eigenvalue pairs.}
The three distinct eigenspaces and their eigenvalue pairs
$(\sigma_1^{(k)},\sigma_2^{(k)})$ for $(S_1, S_2)$ (where $k \in \{ 1,\ldots,8 \}$) are:
\begin{center}
\renewcommand{\arraystretch}{1.4}
\begin{tabular}{llcc}
\toprule
Mode & Eigenvector & $\sigma_1^{(k)}$ & $\sigma_2^{(k)}$ \\
\midrule
Global \;(mult.\ 1) &
  $u_1 = \tfrac{1}{\sqrt{8}}(1,\ldots,1)^\top$ &
  $\varepsilon$ & $\varepsilon$ \\
Between-community \;(mult.\ 1) &
  $u_2 = \tfrac{1}{\sqrt{8}}(1,1,1,1,-1,-1,-1,-1)^\top$ &
  $8+\varepsilon$ & $\varepsilon$ \\
Within-community \;(mult.\ 6) &
  ONB of within-community subspaces &
  $8+\varepsilon$ & $4+\varepsilon$ \\
\bottomrule
\end{tabular}
\end{center}
 
\paragraph{Signed barycenter formulas.} We observe the pair $(G_1, G_2)$ and wish to \emph{extrapolate} the trend from $G_2$ (fragmented) toward $G_1$ (integrated) and beyond. This is formulated as a signed weighted Fr\'{e}chet mean with weights
\[
  \lambda_1 = +2 \quad \text{(anchor on } S_1\text{)}, \qquad
  \lambda_2 = -1 \quad \text{(repulsor from } S_2\text{)}, \qquad
  \lambda_1 + \lambda_2 = 1,
\]
The predicted SPD matrix is $\hat{S} = \operatorname*{arg\,min}_{S \in \mathbb{S}_{++}^8}
  \Bigl[\lambda_1 \cdot d^2(S,\, S_1) + \lambda_2 \cdot d^2(S,\, S_2)\Bigr],$ where $d(\cdot,\cdot)$ denotes the Riemannian distance under the chosen metric.
Since $S_1$ and $S_2$ are commutative, for each mode $k \in \{ 1, \ldots, 8 \}$, the predicted scalar eigenvalue $\hat{s}^{(k)}$ is:
\begin{itemize}
    \item Bures-Wasserstein: $\hat{s}^{(k)}_{BW}
    = \bigl(2\sqrt{\sigma_1^{(k)}} - \sqrt{\sigma_2^{(k)}}\bigr)^2,$
    \item Affine-Invariant and Log-Euclidean: $\hat{s}^{(k)}_{AI/LE}
    = \frac{\bigl(\sigma_1^{(k)}\bigr)^2}{\sigma_2^{(k)}}$.
\end{itemize}

\begin{figure}[ht!]
    \centering
    \includegraphics[width=0.55\linewidth]{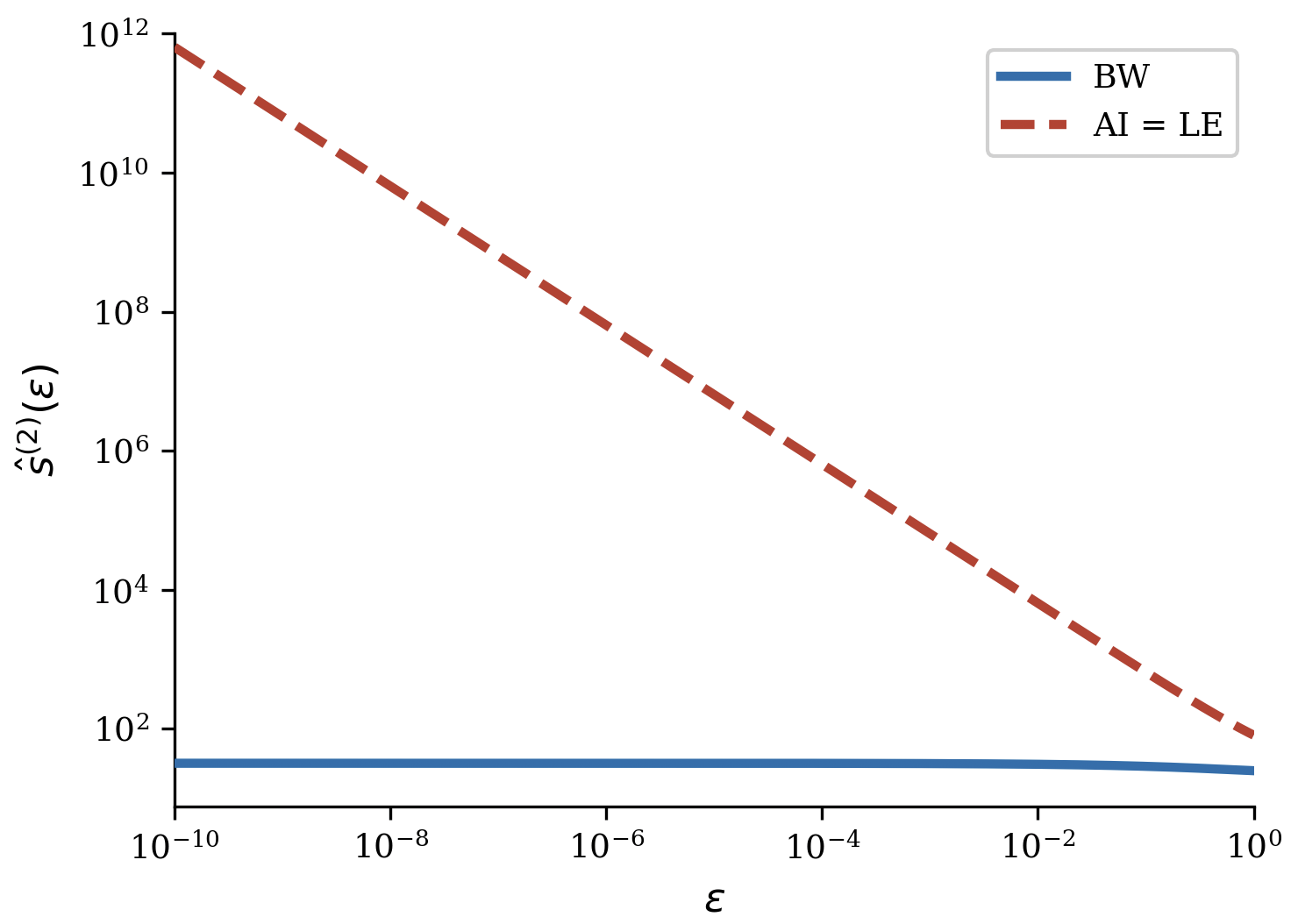}
    \caption{The eigenvalue $\hat{s}^{(2)}$ under three metrics \\
    (Mode 2: $\sigma_1^{(2)}=8 + \varepsilon$ and $\sigma_2^{(2)}=\varepsilon$)}
    \label{fig:mode_2-epsilon-0}
\end{figure}
 
\paragraph{Instability as $\varepsilon\to 0$.}
The between-community mode ($k=2$, with $\sigma_2^{(2)}=\varepsilon$)
is the critical one (Figure~\ref{fig:mode_2-epsilon-0}):
\begin{itemize}
    \item Bures-Wasserstein: $\hat{s}^{(2)}_{BW} = \bigl(2\sqrt{8+\varepsilon} - \sqrt{\varepsilon}\bigr)^2 \;\xrightarrow{\varepsilon\to 0}\; (2\sqrt{8})^2 = 32$,
    \item AI and LE: $\hat{s}^{(2)}_{AI/LE} = \frac{(8+\varepsilon)^2}{\varepsilon} \;\xrightarrow{\varepsilon\to 0}\; +\infty.$
\end{itemize}

The remaining modes are stable for all three metrics
(within-community: BW$\to 36{-}16\sqrt{2}$, AI/LE$\to 16$;
global: both $\to 0$).
The divergence is summarised below.
 
\begin{center}
\renewcommand{\arraystretch}{1.4}
\begin{tabular}{lcccc}
\toprule
Mode & $(\sigma_1^{(k)},\,\sigma_2^{(k)})$ &
  $\hat{s}^{(k)}_{BW}$ &
  $\hat{s}^{(k)}_{AI/LE}$ & Verdict \\
\midrule
Global           & $(\varepsilon,\;\varepsilon)$     & $0$ & $0$   & stable \\
Between-community & $(8{+}\varepsilon,\;\varepsilon)$ & $32$ & $\infty$ &
  \textcolor{red}{\textbf{AI/LE diverge}} \\
Within-community & $(8{+}\varepsilon,\;4{+}\varepsilon)$ &
  $36-16\sqrt{2}$ & $16$ & stable \\
\bottomrule
\end{tabular}
\end{center}
 
\paragraph{Root cause.}
The divergence of AI/LE stems from the singularity of the logarithm at $\sigma_2^{(2)} = \varepsilon \to 0$. With weight $\lambda_2 = -1$, the log-domain contribution $-\log\varepsilon \to +\infty$ causes the predicted eigenvalue to explode. BW avoids this entirely by operating via $\sqrt{\varepsilon}\to 0$ rather than $|\log\varepsilon|\to\infty$, and remains stable even at $\varepsilon = 0$. Consequently, AI and LE require a large regularization $\varepsilon$ to avoid numerical blow-up, which corrupts the Laplacian structure (adding artificial self-loops). BW requires no such compromise. 
\section{Additional visualization} \label{appen:visualization}
\subsection{Ant Social Organization Network} \label{subsec:appen-Network}
\begin{figure*}[ht!]
    \centering
    \includegraphics[width=1\linewidth]{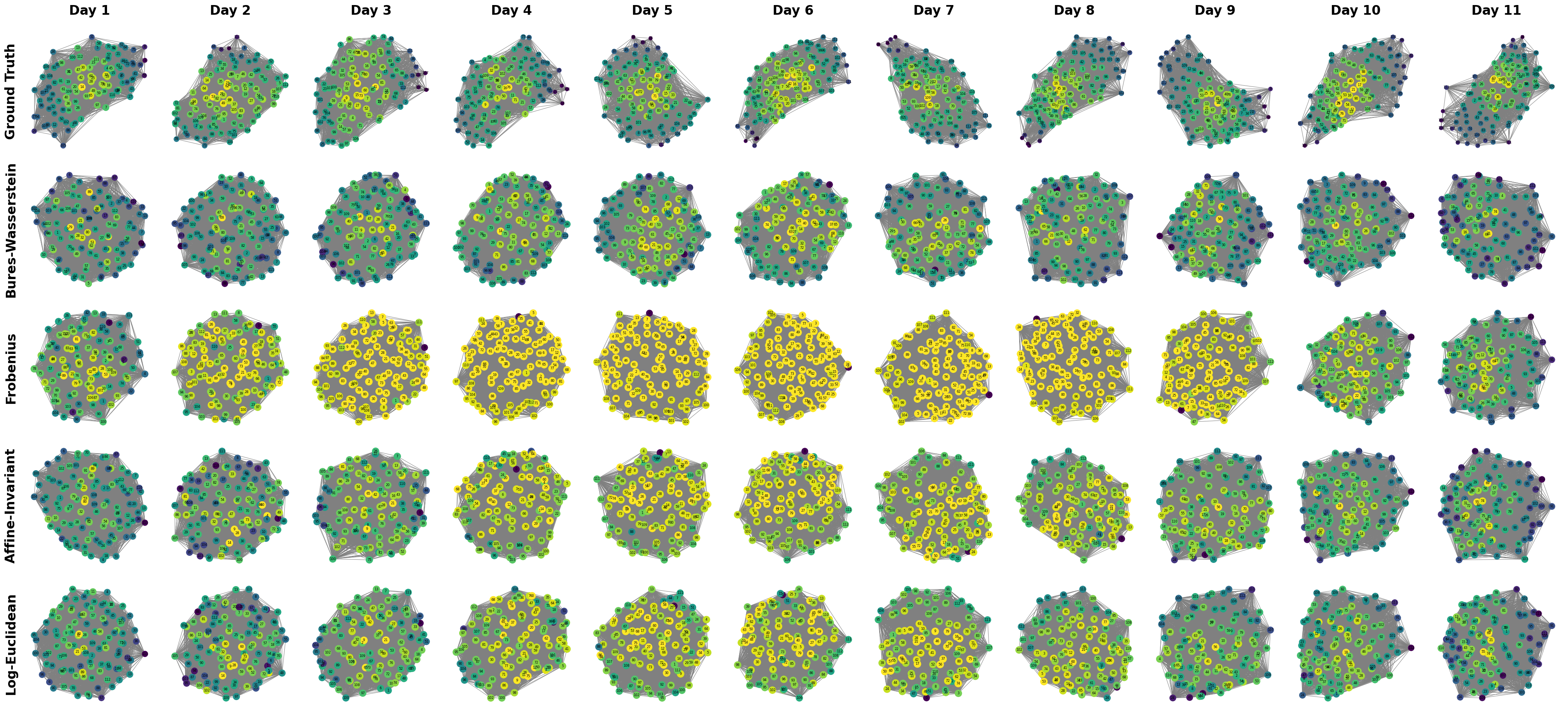}
    \caption{Network plots for ground truth and four methods on 11 days}
    \label{fig:network-plot}
\end{figure*}

Figure \ref{fig:network-plot} illustrates the networks generated by the Wasserstein, Frobenius, Affine-Invariant, and Log-Euclidean regression methods compared against the ground truth. We use a spring-force layout in which each node's color corresponds to its degree.

As observed in the plots, the Frobenius, Affine-Invariant, and Log-Euclidean-based methods, particularly during the interpolation range (i.e., Days 4 to 8), exhibit a ``smoothing'' effect. In these rows, the node colors become largely uniform, indicating that the degree distribution has homogenized and the distinct structural features are lost. In contrast, our proposed Wasserstein method preserves the heterogeneity of the degrees throughout the timeline. The color variation in the Wasserstein row remains distinct and closely mirrors the diverse degree distributions observed in the ground-truth networks.

\subsection{Simple Fréchet Regression Example with Diffusion Tensors}
This figure provides a minimal illustration of Fréchet regression applied to diffusion tensors, modeled as $3 \times 3$ SPD matrices with respect to the Bures–Wasserstein metric. Given two observed diffusion tensors (black), the fitted regression curve corresponds to the unique geodesic connecting them. As the covariate parameter $t$ varies within the observed range $[0,1]$, the predicted tensors (green) smoothly interpolate between the endpoints while preserving positive definiteness. Conversely, when $t$ extends outside this range, the regression produces extrapolated tensors (red) that rapidly become highly anisotropic (flattened) or nearly degenerate. The figure illustrates this anisotropy near the boundary of the SPD cone.
\vspace{-0.3cm}
\begin{figure}[!ht]
    \centering
    \includegraphics[width=.8\linewidth]{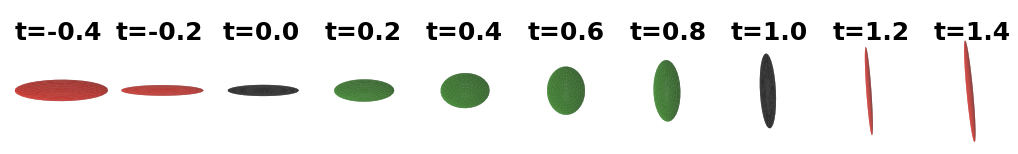}
    \caption{Interpolation (green) and extrapolation (red) between two (black) ellipsoids (3x3 SPD matrices) under the BW metric}
    \label{fig:ellipsoid-geodesic}
\end{figure}
\vspace{-0.3cm}

\subsection{Helix Tensors Prediction (DTI)} \label{subsec:appen-DTI}
\vspace{-0.4cm}
\begin{figure}[!ht]
    \centering
\begin{subfigure}{.46\linewidth}
  \centering
  \includegraphics[width=\linewidth]{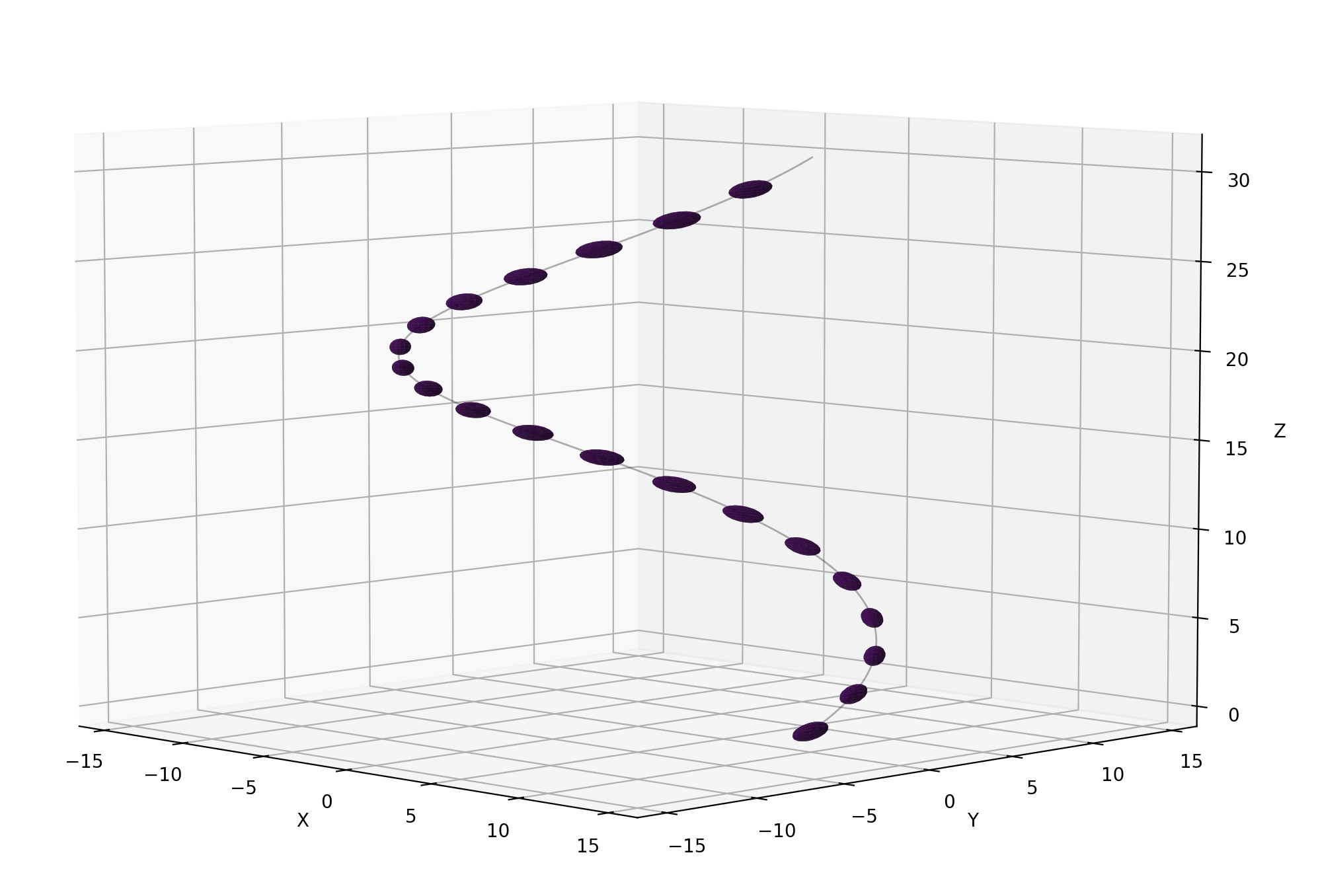}
  \caption{Ground-truth tensors}
  \label{fig:ground-truth-tensor-visualization}
\end{subfigure}
\begin{subfigure}{.47\linewidth}
  \centering
  \includegraphics[width=\linewidth]{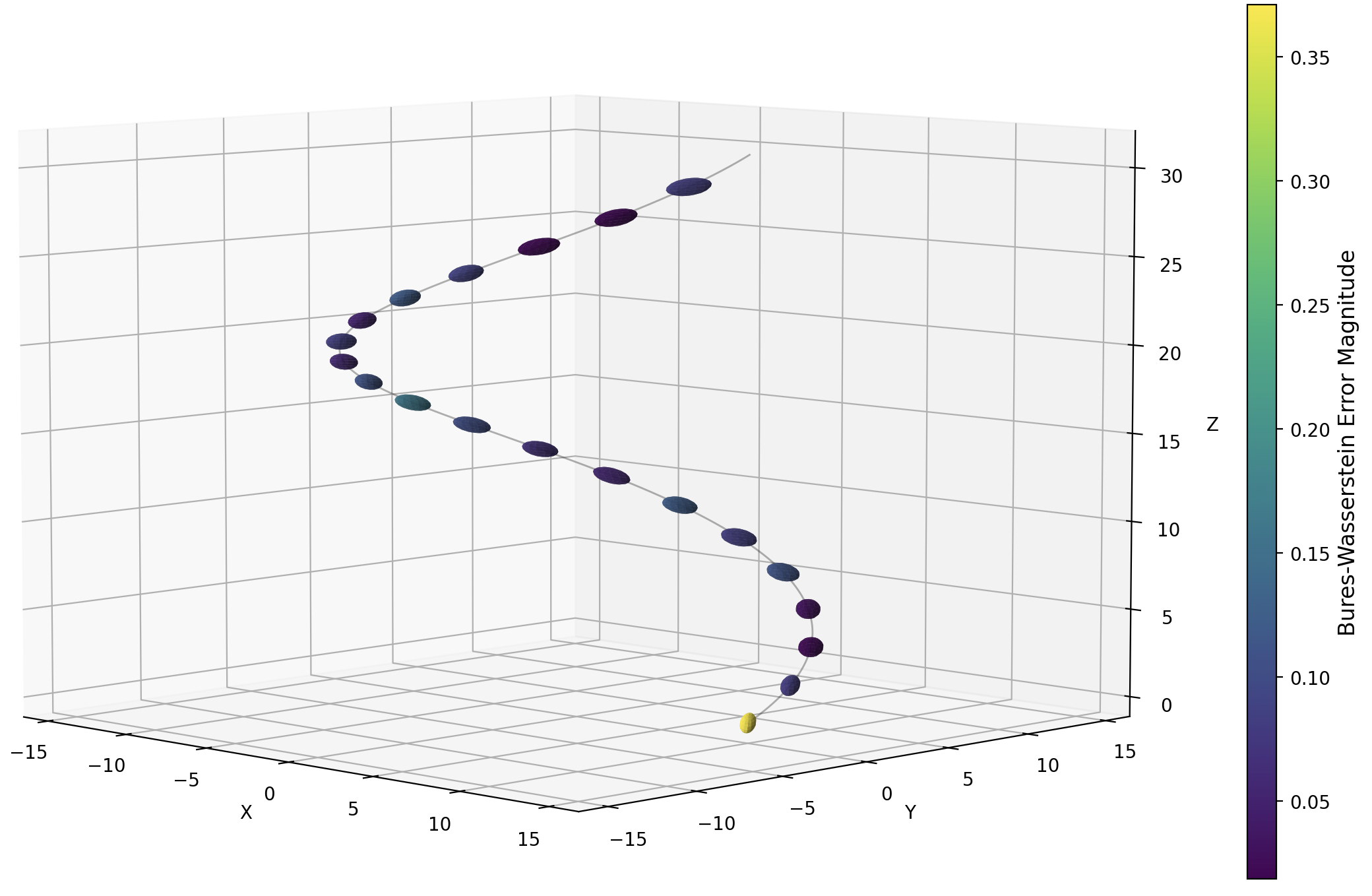}
  \caption{Predicted tensors}
  \label{fig:predicted-tensor-visualization}
\end{subfigure}
    \caption{Helix visualization for 20 ground-truth tensors (from 100,000 samples) and their prediction from the model}
    \label{fig:helix-visualization}
\end{figure}
We visualize the regression results along the helical trajectory in Figure \ref{fig:helix-visualization}. Figure \ref{fig:ground-truth-tensor-visualization} displays 20 representative ground-truth tensors (subsampled from the $100,000$ dataset) along the helical backbone, visualized as ellipsoids. Figure \ref{fig:predicted-tensor-visualization} shows the corresponding predicted tensors resulting from our Fréchet regression model. The color of each ellipsoid in Figure \ref{fig:predicted-tensor-visualization} encodes the Bures-Wasserstein error magnitude relative to the ground truth. Visually, the predicted tensors closely match the orientation and anisotropy of the ground truth. This is quantitatively confirmed by the color scale: most tensors are dark blue, indicating small Bures-Wasserstein distances near 0.05. A notable exception is the yellow ellipsoid at the base of the helix (near $z=0$). This point represents the furthest extrapolation boundary in this set, resulting in a higher error magnitude of approximately 0.35. Overall, these visualizations demonstrate that our R-SGD optimizer effectively recovers the synthetic tensors even along highly non-linear spatial curves. 

\vspace{-0.3cm}
\section{Proofs for main theorems and propositions} \label{appen:proof}
\vspace{-0.3cm}
In this appendix section, we will show in detail the proofs for Theorem \ref{thm:existence}, Proposition~\ref{prop:stationary-bound}, Proposition \ref{prop:no-maximum}, Lemma \ref{prop:upper-bound-curvature}, Proposition \ref{prop:non-projection-iteration}, Lemmas~\ref{lemma:upper}-\ref{lemma:lsmooth}, and Theorem \ref{thm:convergence-rate-bwgd}. First, we come to the proof of our main theorem.

\begin{proof}[\textbf{Proof for Theorem \ref{thm:existence}}]
First, we use the Löwner-Heinz theorem \cite{lowner1934monotone}, which states that the function $f(t) = t^{1/2}$ is operator monotone on $[0, \infty)$. That is, for $A, B \succeq 0$, if $A \preceq B$, then $A^{1/2} \preceq B^{1/2}$. Let $\mathbb{S}_{+}^d = \lbrace \Sigma \in \mathbb{R}^{d \times d}: \Sigma^\top = \Sigma \text{ and } \Sigma \succeq 0\rbrace$ be the set of positive semi-definite matrices. For convenience, we use some notations: 
$$A_{\min} = \sum_{i \in \mathcal{I}} \lambda_i^+ \sqrt{\lambda_{\min}(\Sigma_i)}, \quad B_{\min} = \sum_{j \in \mathcal{J}} \lambda_j^- \sqrt{\lambda_{\min}(\Sigma_j)},$$ 
$$A_{\max} = \sum_{i \in \mathcal{I}} \lambda_i^+ \sqrt{\lambda_{\max}(\Sigma_i)}, \quad B_{\max} = \sum_{j \in \mathcal{J}} \lambda_j^- \sqrt{\lambda_{\max}(\Sigma_j)}.$$

We proceed in two steps.

\noindent\textbf{Step 1: Coercivity and existence on the closed cone $\mathbb S_+^d$.}
For each $k \in \lbrace 1,...,n \rbrace$, from $\lambda_{\min}(\Sigma_k)I \preceq \Sigma_k \preceq \lambda_{\max}(\Sigma_k)I$ and operator monotonicity,
\[
\sqrt{\lambda_{\min}(\Sigma_k)} \,S^{1/2}\ \preceq\ (S^{1/2}\Sigma_k S^{1/2})^{1/2}\ \preceq\ \sqrt{\lambda_{\max}(\Sigma_k)} \,S^{1/2}.
\]
Taking traces and separating positive/negative weights gives
\begin{align*}
\sum_{k=1}^n \lambda_k\,\mathrm{Tr}\!\big((S^{1/2}\Sigma_k S^{1/2})^{1/2}\big)
\ &\leq \left( \sum_{i \in \mathcal{I}}\lambda_i^+ \sqrt{\lambda_{\max}(\Sigma_i)} - \sum_{j \in \mathcal{J}}\lambda_j^- \sqrt{\lambda_{\min}(\Sigma_j)} \right)  \,\mathrm{Tr}(S^{1/2})\\
&= (A_{\max} - B_{\min})\mathrm{Tr}(S^{1/2}).
\end{align*}
Therefore,
\begin{equation}\label{eq:coercive-bound}
F(S)\ \ge\ \mathrm{Tr}S\ -\ 2(A_{\max} - B_{\min})   \,\mathrm{Tr}(S^{1/2})
\ +\ \sum_{k=1}^n \lambda_k\mathrm{Tr}\Sigma_k.
\end{equation}
Let $\lambda'_1,\ldots,\lambda'_d$ be the eigenvalues of $S$. By the Cauchy-Schwarz inequality,
$$\mathrm{Tr}(S^{1/2})=\sum_{l=1}^d \sqrt{\lambda'_l}\le \sqrt{d\,\sum_{l =1}^d \lambda'_l}=\sqrt{d\,\mathrm{Tr}S}.$$ Setting
$x:=\mathrm{Tr}S$ and $c_1:=A_{\max} - B_{\min}$, the right–hand side of \eqref{eq:coercive-bound} is
$$
x - 2c_1\sqrt{dx} + \text{const} \;=\; \big(\sqrt x - c_1\sqrt d\big)^2 - c_1^2 d + \text{const}\ \xrightarrow[x\to\infty]{}\ +\infty.
$$
Hence $F$ is \emph{coercive} on the closed cone $\mathbb S_+^d$. Because $F$ is continuous, it attains a minimum on $\mathbb S_+^d$ (Weierstrass on closed, bounded sublevel sets). Denote one such minimizer by $\widehat S\succeq 0$.

\noindent\textbf{Step 2: Under \eqref{eq:SNM}, no singular matrix minimizes $F$.} Assume, for contradiction, that the minimizer $\widehat S$ is singular.
Let $p:=\dim\ker\widehat S\ge1$. For $t>0$, we define
$$ S(t)\ :=\ \widehat S + tI\ \in\ \mathbb{S}_{++}^d. $$
First, we bound the Euclidean gradient of $F$ on $\mathbb{S}_{++}^d$. Let $S \in \mathbb{S}_{++}^d$. By Appendix~\ref{appen:prelim}, the Euclidean gradient of $W_2^2(\cdot,\Sigma_k)$ at $S$ is $I - \mathrm{GM}(S^{-1},\Sigma_k)$. Since $\sum_{k=1}^n\lambda_k = 1$, we have
$$\nabla F(S) = I - \sum_{k=1}^n \lambda_k \mathrm{GM}(S^{-1},\Sigma_k) = I - S^{-1/2}\Big(\sum_{k=1}^n \lambda_k \big(S^{1/2}\Sigma_k S^{1/2}\big)^{1/2}\Big)S^{-1/2}.$$
We apply the lower bound of Step~1 to the terms with positive weights and the upper bound of Step~1 to the terms with negative weights. This gives
$$\sum_{k=1}^n \lambda_k \big(S^{1/2}\Sigma_k S^{1/2}\big)^{1/2} \succeq (A_{\min} - B_{\max})\,S^{1/2}.$$
Let $c_2 := A_{\min} - B_{\max}$, which is positive under~\eqref{eq:SNM}. Conjugating both sides by $S^{-1/2}$, we obtain
$$\nabla F(S) \preceq I - c_2\, S^{-1/2}.$$

Next, the function $t \mapsto F(S(t))$ is differentiable on $(0,\infty)$, and its derivative is $\mathrm{Tr}\,\nabla F(S(t))$. By the bound above,
$$\frac{d}{dt}F(S(t)) \leq d - c_2\,\mathrm{Tr}\big(S(t)^{-1/2}\big).$$
The matrix $S(t)$ has the eigenvalue $t$ with multiplicity $p$, and all its eigenvalues are positive. Hence, $\mathrm{Tr}(S(t)^{-1/2}) \geq p/\sqrt{t}$, and
$$\frac{d}{dt}F(S(t)) \leq d - \frac{c_2\,p}{\sqrt{t}}.$$
Fix $t>0$ and $\varepsilon \in (0,t)$. Integrating this inequality over $[\varepsilon, t]$ gives
$$F(S(t)) - F(S(\varepsilon)) \leq d\,(t-\varepsilon) - 2c_2\,p\,\big(\sqrt{t} - \sqrt{\varepsilon}\big).$$
Since $F$ is continuous on $\mathbb{S}_+^d$ and $S(\varepsilon) \to \widehat S$ as $\varepsilon \to 0$, we obtain
$$F(S(t)) - F(\widehat S) \leq d\,t - 2c_2\,p\sqrt{t} = \sqrt{t}\,\big(d\sqrt{t} - 2c_2\,p\big).$$
For every $t \in \big(0, (2c_2p/d)^2\big)$, the right-hand side is negative. Thus, $F(S(t)) < F(\widehat S)$, which contradicts the minimality of the singular $\widehat S$ on $\mathbb S_+^d$.

\medskip
\noindent Since a minimizer on $\mathbb S_+^d$ exists (Step~1) and no singular matrix can be optimal under \eqref{eq:SNM} (Step~2), the minimum is attained at some $S^\star\in \mathbb{S}_{++}^d$.
\end{proof}

Next, we provide the bound for all the eigenvalues of every stationary point $S_\ast$, if one exists.

\begin{proof}[\textbf{Proof for Proposition~\ref{prop:stationary-bound}}]
We split the stationary equation \eqref{eqn:stationary-equation} into its positive and negative components:
\begin{align} \label{eq:split_stationary}
S_* = \sum_{i\in \mathcal{I}} \lambda_i^+ \left( S_*^{1/2} \Sigma_i S_*^{1/2} \right)^{1/2} - \sum_{j \in \mathcal{J}} \lambda_j^- \left( S_*^{1/2} \Sigma_j S_*^{1/2} \right)^{1/2}
\end{align}
For any SPD matrix $\Sigma_k$, the Löwner order gives $\lambda_{\min}(\Sigma_k) I \preceq \Sigma_k \preceq \lambda_{\max}(\Sigma_k) I$. Conjugating by $S_*^{1/2}$ and applying the Löwner-Heinz theorem \cite{lowner1934monotone} yields:

\begin{align} \label{eq:lowner_bounds}
\sqrt{\lambda_{\min}(\Sigma_k)} S_*^{1/2} \preceq \left( S_*^{1/2} \Sigma_k S_*^{1/2} \right)^{1/2} \preceq \sqrt{\lambda_{\max}(\Sigma_k)} S_*^{1/2}
\end{align}

To bound $S_*$ from above, we substitute the upper bounds for the positive terms in \eqref{eq:split_stationary} and the lower bounds for the subtracted negative terms:
\begin{align*}
S_* \preceq \Big( \sum_{i\in \mathcal{I}} \lambda_i^+ \sqrt{\lambda_{\max}(\Sigma_i)} - \sum_{j\in \mathcal{J}} \lambda_j^- \sqrt{\lambda_{\min}(\Sigma_j)} \Big) S_*^{1/2} := C_{\text{upper}} S_*^{1/2}
\end{align*}
Similarly, bounding $S_*$ from below by minimizing the positive terms and maximizing the subtracted negative terms yields:
\begin{align*}S_* \succeq \Big( \sum_{i\in \mathcal{I}} \lambda_i^+ \sqrt{\lambda_{\min}(\Sigma_i)} - \sum_{j\in \mathcal{J}} \lambda_j^- \sqrt{\lambda_{\max}(\Sigma_j)} \Big) S_*^{1/2} := C_{\text{lower}} S_*^{1/2}
\end{align*}
Note that $C_{\text{lower}} > 0$ is guaranteed by the Spectral Dominance condition in Theorem~\ref{thm:existence}. Finally, for any scalar $C > 0$, the condition $S_* \preceq C S_*^{1/2}$ implies that every eigenvalue $\lambda$ of $S_*$ satisfies $\lambda \le C \sqrt{\lambda}$, which simplifies to $\lambda \le C^2$, or $S_* \preceq C^2 I$.

Applying this eigenvalue logic to both $C_{\text{lower}}$ and $C_{\text{upper}}$, we have $C_{\text{lower}}^2 I \preceq S_* \preceq C_{\text{upper}}^2 I$. 
\end{proof}
\vspace{-0.3cm}
Next, we further show that under our existence condition, there is no local maximum in $\mathbb{S}_{++}^d$.
\vspace{-0.3cm}
\begin{proof}[\textbf{Proof for Proposition \ref{prop:no-maximum}}]
Let $S_*$ be a stationary point of $F(S)$. Let $1>\varepsilon>0$. Consider $S_t := t S_*$, for $t \in (1- \varepsilon, 1+\varepsilon)$, such that $S_t \in \mathbb{S}_{++}^d$ for all $t$.
We have
\vspace{-0.3cm}
\begin{align*}
    F(S_t)  &= \sum_{k=1}^n \lambda_k W^2_2(S_t,\Sigma_k) \\
            &= \sum_{k=1}^n \lambda_k \left( \mathrm{Tr}(\Sigma_k) + \mathrm{Tr}(S_t) - 2\mathrm{Tr}\left( S_t^{1/2} \Sigma_k S_t^{1/2} \right)^{1/2}\ \right) \\
            &= \sum_{k=1}^n \lambda_k \mathrm{Tr}(\Sigma_k) + \mathrm{Tr}(S_t) - 2 \mathrm{Tr} \left( \sum_k \lambda_k  \left( S_t^{1/2} \Sigma_k S_t^{1/2} \right)^{1/2} \right) \\
            &= \sum_{k=1}^n \lambda_k \mathrm{Tr}(\Sigma_k) + t\, \mathrm{Tr}(S_*) - 2 \sqrt{t} \,\mathrm{Tr} \left( \sum_k \lambda_k  \left( S_*^{1/2} \Sigma_k S_*^{1/2} \right)^{1/2} \right) \\
            &= \sum_{k=1}^n \lambda_k \mathrm{Tr}(\Sigma_k) + t\, \mathrm{Tr}(S_*) - 2 \sqrt{t} \,\mathrm{Tr} (S_*) \quad (\text{From equation (\ref{eqn:stationary-equation})}) \\
            &= \sum_{k=1}^n \lambda_k \mathrm{Tr}(\Sigma_k) + (t - 2\sqrt{t})\, \mathrm{Tr}(S_*).
\end{align*}
Denote $f(t) := t - 2\sqrt{t}$ on $(1- \varepsilon, 1+\varepsilon)$. We have  $f'(t) = \frac{\sqrt{t} - 1}{\sqrt{t}}.$ From that, $f'(1)=0$, $f'(t) < 0$ on $(1-\varepsilon,1)$, and $f'(t) > 0$ on $(1,1+\varepsilon)$. Thus, $f(t)$ has a local minimum at $t=1$ on $(1- \varepsilon, 1+\varepsilon)$. Since $\mathrm{Tr}(S_*) > 0$, then $F$ also has a local minimum at $S_*$ on the smooth curve $S_t$ for $t \in (1- \varepsilon, 1+\varepsilon)$. Therefore, $S_*$ cannot be a local maximum.  
\end{proof}

Before proving Lemma \ref{prop:upper-bound-curvature}, we refer to \citep[Theorem 1.1]{takatsu2010wasserstein} for the formula of the sectional curvature of $\mathbb{S}_{++}^d$.

\begin{Theorem}[Sectional curvatures \protect{\citep[Theorem 1.1]{takatsu2010wasserstein}}] \label{thm:section-curvature}
Let $P$ be an orthogonal matrix and let $\{\lambda_i\}_{i=1}^d$ be positive numbers.
Set $V = P \, \mathrm{diag}(\lambda_1,\dots,\lambda_d)\, P^{\top}$.
Let $\{E_{ij}\}$ denote the matrix units.  
A basis of $T_V\mathcal{N}^d$, the tangent space of the Bures--Wasserstein manifold at $V$, is given by
\begin{align*}
e_i^+ = P (E_{ii} + E_{dd}) P^{\top},
e_{ij} = \frac{1}{\sqrt{\lambda_i+\lambda_j}}\,
         P(E_{ii}-E_{jj})P^{\top}, 
f_{ij} = \frac{1}{\sqrt{\lambda_i+\lambda_j}}\,
         P(E_{ij}+E_{ji})P^{\top}, 
\end{align*}
for $1\le i<j\le d$. Then the sectional curvature $K(\cdot,\cdot)$ satisfies
\[
\begin{aligned}
(1)\quad &K(e_i^+, e_{ij}) = 0,\\[4pt]
(2)\quad &K(e_i^+, e_j^+) = 0,\\[4pt]
(3)\quad &K(e_i^+, f_{ij}) =
\frac{3\,\lambda_i \lambda_j}{
      (\lambda_i+\lambda_j)^2(\lambda_i+\lambda_d)},\\[8pt]
(4)\quad &K(e_i^+, f_{kl}) = 0,
\qquad (1<k<l<d),\\[4pt]
(5)\quad &K(e_{ij}, e_{kl}) = 0,
\qquad (\{i,j\}\cap\{k,l\}=\varnothing),\\[4pt]
(6)\quad &K(e_{ij}, f_{kl}) = 0,
\qquad (j\ne k),\\[4pt]
(7)\quad &K(e_{ik}, f_{ij}) =
\frac{3\,\lambda_i \lambda_j}{
      (\lambda_i+\lambda_j)^2(\lambda_i+\lambda_k)},
\qquad (j\ne k),\\[8pt]
(8)\quad &K(e_{ij}, f_{ij}) =
\frac{12\,\lambda_i\lambda_j}{
      (\lambda_i+\lambda_j)^3},\\[10pt]
(9)\quad &K(f_{ij}, f_{kl}) = 0,
\qquad (\{i,j\}\cap\{k,l\}=\varnothing),\\[4pt]
(10)\quad &K(f_{ij}, f_{ik}) =
\frac{3\,\lambda_j\lambda_k}{
     (\lambda_i+\lambda_j)(\lambda_j+\lambda_k)(\lambda_k+\lambda_i)}.
\end{aligned}
\]
\end{Theorem}

\begin{proof}[\textbf{Proof for Lemma \ref{prop:upper-bound-curvature}}] 
    From Theorem~\ref{thm:section-curvature}, consider $\Sigma \in \mathbb{S}_{++}^d$ and its smallest eigenvalue $\lambda_{\min}(\Sigma)$, then we can have an upper bound for all sectional curvatures at point $\Sigma$ as
    \begin{align*}
        K( \cdot,\cdot) \underset{(8)}{\leq} K(e_{ij},f_{ij}) = \frac{12 \lambda_i \lambda_j}{ (\lambda_i + \lambda_j)^3} &\leq \frac{12 \lambda_i \lambda_j}{ (2 \sqrt{\lambda_i \lambda_j})(2 \sqrt{\lambda_i \lambda_j})(\lambda_i + \lambda_j)} \quad \\
        & (\text{AM-GM inequality}) \\
        & = \frac{3}{\lambda_i + \lambda_j} \\
        & \leq \frac{3}{2 \lambda_{\min}(\Sigma)}.
    \end{align*}
    Thus, for each $\Sigma {\in} \mathbb{S}_{++}^d$, $K^+(\Sigma) {:=} \frac{3}{2 \lambda_{\min}(\Sigma)}$ is an upper bound for all sectional curvatures at $\Sigma$.
\end{proof}

Now, we prove that all iterations of Algorithm \ref{alg:bwgd} stay inside the domain under our condition.

\begin{proof}[\textbf{Proof for Proposition \ref{prop:non-projection-iteration}}]
Let $\Sigma_k \in \mathbb{S}_{++}^d, \lambda_k \in \mathbb{R}$, for $k \in \lbrace 1,...,n\rbrace$, satisfying condition (\ref{eq:SNM}). For any $S \in \mathbb{S}_{++}^d$, 
\begin{align*}
    \sum_{k=1}^n \lambda_k \mathrm{GM}(S^{-1}, \Sigma_k) &= \sum_{i \in \mathcal{I}} \lambda_i^+ \mathrm{GM}(S^{-1}, \Sigma_i) - \sum_{j \in \mathcal{J}}\lambda_j^- \mathrm{GM}(S^{-1},\Sigma_j) \\
    &= \sum_{i \in \mathcal{I}} \lambda_i^+ S^{-1/2}(S^{1/2}\Sigma_i S^{1/2})^{1/2}S^{-1/2} - \sum_{j \in \mathcal{J}} \lambda_j^- S^{-1/2}(S^{1/2}\Sigma_j S^{1/2})^{1/2}S^{-1/2} \\
    &= S^{-1/2} \left( \sum_{i \in \mathcal{I}} \lambda_i^+ (S^{1/2}\Sigma_i S^{1/2})^{1/2} - \sum_{j \in \mathcal{J}} \lambda_j^- (S^{1/2}\Sigma_j S^{1/2})^{1/2} \right) S^{-1/2}
\end{align*}
Denote $M:= \sum_{i \in \mathcal{I}} \lambda_i^+ (S^{1/2}\Sigma_i S^{1/2})^{1/2} - \sum_{j \in \mathcal{J}} \lambda_j^- (S^{1/2}\Sigma_j S^{1/2})^{1/2}$. We will show that $M \succ 0$.

First, consider the positive weight terms. Since $\Sigma_i \succeq \lambda_{\min}(\Sigma_i)I$, conjugating by $S^{1/2}$ gives $S^{1/2}\Sigma_i S^{1/2} \succeq \lambda_{\min}(\Sigma_i) S$. Applying the Löwner-Heinz theorem:
\begin{equation*} \label{eq:pos_bound}
    (S^{1/2}\Sigma_i S^{1/2})^{1/2} \succeq \sqrt{\lambda_{\min}(\Sigma_i)} S^{1/2}.
\end{equation*}

Similarly, for the negative weight terms, we have $\Sigma_j \preceq \lambda_{\max}(\Sigma_j)I$, which implies $S^{1/2}\Sigma_j S^{1/2} \preceq \lambda_{\max}(\Sigma_j) S$. Applying the theorem again:
\begin{equation*} \label{eq:neg_bound}
    (S^{1/2}\Sigma_j S^{1/2})^{1/2} \preceq \sqrt{\lambda_{\max}(\Sigma_j)} S^{1/2}.
\end{equation*}

Substituting these bounds back into the expression for $M$:
\begin{align*}
    M &= \sum_{i \in \mathcal{I}} \lambda_i^+ (S^{1/2}\Sigma_i S^{1/2})^{1/2} - \sum_{j \in \mathcal{J}} \lambda_j^- (S^{1/2}\Sigma_j S^{1/2})^{1/2} \\
    &\succeq \sum_{i \in \mathcal{I}} \lambda_i^+ \left( \sqrt{\lambda_{\min}(\Sigma_i)} S^{1/2} \right) - \sum_{j \in \mathcal{J}} \lambda_j^- \left( \sqrt{\lambda_{\max}(\Sigma_j)} S^{1/2} \right) \\
    &= \left( \sum_{i \in \mathcal{I}} \lambda_i^+ \sqrt{\lambda_{\min}(\Sigma_i)} - \sum_{j \in \mathcal{J}} \lambda_j^- \sqrt{\lambda_{\max}(\Sigma_j)} \right) S^{1/2}.
\end{align*}
Let $\Delta := \sum_{i \in \mathcal{I}} \lambda_i^+ \sqrt{\lambda_{\min}(\Sigma_i)} - \sum_{j \in \mathcal{J}} \lambda_j^- \sqrt{\lambda_{\max}(\Sigma_j)}$. By the hypothesis, we have $\Delta > 0$. Since $S \in \mathbb{S}_{++}^d$, its square root $S^{1/2}$ is also strictly positive definite. Therefore, $M \succeq \Delta S^{1/2} \succ 0$.

Finally, since $M \succ 0$ and $S^{-1/2}$ is invertible, the congruence transformation $S^{-1/2} M S^{-1/2}$ preserves positive definiteness. Thus,
$ \sum_{k=1}^n \lambda_k \mathrm{GM}(S^{-1}, \Sigma_k) \in \mathbb{S}_{++}^d. $
\end{proof}

Next, we prove Lemmas~\ref{lemma:lower} and~\ref{lemma:lsmooth}. Lemma~\ref{lemma:upper} is the $1$-smoothness inequality of~\cite{chewi2020gradient}. Throughout, we fix $X,\Sigma \in \mathbb{S}_{++}^d$ and let $G := \frac{1}{2}W_2^2(\cdot,\Sigma)$. For $Y \in \mathbb{S}_{++}^d$, we write $T_Y := \mathrm{GM}(X^{-1},Y)$. In particular, $T_\Sigma = \mathrm{GM}(X^{-1},\Sigma)$. Recall from Appendix~\ref{appen:prelim} that $T_Y$ is SPD, that $T_Y X T_Y = Y$, and that $\log_X Y = (T_Y - I)X + X(T_Y - I)$. We first derive two identities that are used in all three proofs.

\textbf{Distance identity.} Expanding the product and using $T_Y X T_Y = Y$, we have
\[
\mathrm{Tr}\big((T_Y - I)X(T_Y - I)\big) = \mathrm{Tr}(Y) - 2\,\mathrm{Tr}(T_Y X) + \mathrm{Tr}(X).
\]
By the definition of $T_Y$ and the cyclic property of the trace,
\[
\mathrm{Tr}(T_Y X) = \mathrm{Tr}\big(X^{-1/2}(X^{1/2}YX^{1/2})^{1/2}X^{1/2}\big) = \mathrm{Tr}\big((X^{1/2}YX^{1/2})^{1/2}\big).
\]
Comparing with the formula for $W_2^2$ in Appendix~\ref{appen:prelim}, we obtain
\begin{equation}\label{eq:W2-TY}
W_2^2(X,Y) = \mathrm{Tr}\big((T_Y - I)X(T_Y - I)\big).
\end{equation}

\textbf{Inner-product identity.} By Appendix~\ref{appen:prelim}, the Euclidean gradient of $G$ at $X$ is $\frac{1}{2}(I - T_\Sigma)$. Hence,
\[
\nabla_{\mathrm{bw}}G(X) = 4\big\{\tfrac{1}{2}(I - T_\Sigma)X\big\}_{\mathrm{s}} = (I - T_\Sigma)X + X(I - T_\Sigma).
\]
This equation shows that $I - T_\Sigma$ solves the Lyapunov equation with right-hand side $\nabla_{\mathrm{bw}}G(X)$. Since this solution is unique, $\mathcal{L}_X[\nabla_{\mathrm{bw}}G(X)] = I - T_\Sigma$. By the definition of the Riemannian metric,
\begin{equation}\label{eq:inner-trace}
\langle \nabla_{\mathrm{bw}}G(X), \log_X Y\rangle_X = \tfrac{1}{2}\mathrm{Tr}\big((I - T_\Sigma)\log_X Y\big).
\end{equation}

\begin{proof}[\textbf{Proof for Lemma~\ref{lemma:lower}}]
We first prove the bound on $\lambda_{\max}(T_\Sigma)$. Since $\Sigma \preceq \lambda_{\max}(\Sigma)I$, conjugating by $X^{1/2}$ gives $X^{1/2}\Sigma X^{1/2} \preceq \lambda_{\max}(\Sigma)X$. By the Löwner-Heinz theorem~\cite{lowner1934monotone},
\[
(X^{1/2}\Sigma X^{1/2})^{1/2} \preceq \sqrt{\lambda_{\max}(\Sigma)}\,X^{1/2}.
\]
Conjugating both sides by $X^{-1/2}$ gives $T_\Sigma \preceq \sqrt{\lambda_{\max}(\Sigma)}\,X^{-1/2}$. Since $X^{-1/2} \preceq \lambda_{\min}(X)^{-1/2}I$, we conclude that
\[
\lambda_{\max}(T_\Sigma) \leq \sqrt{\lambda_{\max}(\Sigma)/\lambda_{\min}(X)}.
\]
We now prove the quadratic bound in four steps.

\textbf{Step 1: a lower bound on $G(Y)$.} Let $\pi$ be any coupling of $\mathcal{N}(0,Y)$ and $\mathcal{N}(0,\Sigma)$. Let $\varphi,\psi:\mathbb{R}^d \to \mathbb{R}$ be functions such that $\varphi(y) + \psi(z) \leq \frac{1}{2}\lVert y - z\rVert^2$ for all $y,z \in \mathbb{R}^d$. Integrating this inequality with respect to $\pi$ gives
\[
\frac{1}{2}\int \lVert y - z\rVert^2\, d\pi(y,z) \geq \int \varphi\, d\mathcal{N}(0,Y) + \int \psi\, d\mathcal{N}(0,\Sigma).
\]
Taking the infimum over all couplings $\pi$ on the left-hand side, we obtain
\begin{equation}\label{eq:weak-duality}
G(Y) \geq \int \varphi\, d\mathcal{N}(0,Y) + \int \psi\, d\mathcal{N}(0,\Sigma).
\end{equation}
We choose $\varphi(y) := \frac{1}{2}y^\top(I - T_\Sigma)y$ and $\psi(z) := \frac{1}{2}z^\top(I - T_\Sigma^{-1})z$. These functions satisfy the required constraint, because, for all $y,z \in \mathbb{R}^d$,
\[
\tfrac{1}{2}\lVert y - z\rVert^2 - \varphi(y) - \psi(z) = \tfrac{1}{2}\big\lVert T_\Sigma^{1/2}y - T_\Sigma^{-1/2}z\big\rVert^2 \geq 0.
\]

For a Gaussian $\mathcal{N}(0,C)$ and a symmetric matrix $A$, we have $\int y^\top A y\, d\mathcal{N}(0,C) = \mathrm{Tr}(AC)$. Indeed, if $y \sim \mathcal{N}(0,C)$, then $\mathbb{E}[yy^\top] = C$, and hence
$$\int y^\top A y\, d\mathcal{N}(0,C) = \mathbb{E}\big[\mathrm{Tr}(y^\top A y)\big] = \mathbb{E}\big[\mathrm{Tr}(A\,yy^\top)\big] = \mathrm{Tr}\big(A\,\mathbb{E}[yy^\top]\big) = \mathrm{Tr}(AC).$$

Therefore,~\eqref{eq:weak-duality} becomes
\begin{equation}\label{eq:G-lower}
G(Y) \geq \tfrac{1}{2}\mathrm{Tr}\big((I - T_\Sigma)Y\big) + \tfrac{1}{2}\mathrm{Tr}\big((I - T_\Sigma^{-1})\Sigma\big).
\end{equation}

\textbf{Step 2: equality at $Y = X$.} At $Y = X$, the right-hand side of~\eqref{eq:G-lower} equals
\[
\tfrac{1}{2}\big(\mathrm{Tr}(X) + \mathrm{Tr}(\Sigma) - \mathrm{Tr}(T_\Sigma X) - \mathrm{Tr}(T_\Sigma^{-1}\Sigma)\big).
\]
Since $\Sigma = T_\Sigma X T_\Sigma$, we have $\mathrm{Tr}(T_\Sigma^{-1}\Sigma) = \mathrm{Tr}(XT_\Sigma) = \mathrm{Tr}(T_\Sigma X)$. Moreover, as in the derivation of~\eqref{eq:W2-TY}, $\mathrm{Tr}(T_\Sigma X) = \mathrm{Tr}((X^{1/2}\Sigma X^{1/2})^{1/2})$. Therefore, at $Y = X$, the right-hand side of~\eqref{eq:G-lower} equals $\frac{1}{2}W_2^2(X,\Sigma) = G(X)$.

\textbf{Step 3: subtracting the case $Y = X$.} By Step~2, subtracting $G(X)$ from both sides of~\eqref{eq:G-lower} gives
\[
G(Y) - G(X) \geq \tfrac{1}{2}\mathrm{Tr}\big((I - T_\Sigma)(Y - X)\big).
\]
Since $Y = T_Y X T_Y$, writing $T_Y = (T_Y - I) + I$ and expanding gives
\[
Y - X = (T_Y - I)X + X(T_Y - I) + (T_Y - I)X(T_Y - I) = \log_X Y + (T_Y - I)X(T_Y - I).
\]
We substitute this expression into the previous inequality. By~\eqref{eq:inner-trace}, the part involving $\log_X Y$ equals $\langle \nabla_{\mathrm{bw}}G(X), \log_X Y\rangle_X$. Hence,
\begin{equation}\label{eq:G-lower-2}
G(Y) - G(X) \geq \langle \nabla_{\mathrm{bw}}G(X), \log_X Y\rangle_X + \tfrac{1}{2}\mathrm{Tr}\big((I - T_\Sigma)(T_Y - I)X(T_Y - I)\big).
\end{equation}

\textbf{Step 4: bounding the last term.} The matrix $(T_Y - I)X(T_Y - I)$ is positive semidefinite, and $I - T_\Sigma \succeq (1 - \lambda_{\max}(T_\Sigma))I$. For a positive semidefinite matrix $P$ and a symmetric matrix $A \succeq cI$, we have $\mathrm{Tr}(AP) \geq c\,\mathrm{Tr}(P)$, since $\mathrm{Tr}((A - cI)P) = \mathrm{Tr}(P^{1/2}(A - cI)P^{1/2}) \geq 0$. Applying this fact and then~\eqref{eq:W2-TY},
\[
\mathrm{Tr}\big((I - T_\Sigma)(T_Y - I)X(T_Y - I)\big) \geq \big(1 - \lambda_{\max}(T_\Sigma)\big)W_2^2(X,Y).
\]
Since $1 - a \geq -(a - 1)_+$ for every $a \in \mathbb{R}$ and $W_2^2(X,Y) \geq 0$, the right-hand side is at least $-(\lambda_{\max}(T_\Sigma) - 1)_+W_2^2(X,Y)$. Substituting this bound into~\eqref{eq:G-lower-2} gives
\[
G(Y) - G(X) \geq \langle \nabla_{\mathrm{bw}}G(X), \log_X Y\rangle_X - \tfrac{1}{2}\big(\lambda_{\max}(T_\Sigma) - 1\big)_+W_2^2(X,Y).
\]
Multiplying both sides by $-1$ gives the claim.
\end{proof}

\begin{proof}[\textbf{Proof for Lemma~\ref{lemma:lsmooth}}]
For $k \in \lbrace 1,...,n\rbrace$, let $G_k := \frac{1}{2}W_2^2(\cdot,\Sigma_k)$. Separating positive and negative weights, we write
\[
\Phi = \sum_{i\in\mathcal{I}}\lambda_i^+G_i - \sum_{j\in\mathcal{J}}\lambda_j^-G_j.
\]
Since the Bures-Wasserstein gradient is linear,
\[
\nabla_{\mathrm{bw}}\Phi(X) = \sum_{i\in\mathcal{I}}\lambda_i^+\nabla_{\mathrm{bw}}G_i(X) - \sum_{j\in\mathcal{J}}\lambda_j^-\nabla_{\mathrm{bw}}G_j(X).
\]

For each $i \in \mathcal{I}$, we apply Lemma~\ref{lemma:upper} with $\Sigma = \Sigma_i$ and multiply by $\lambda_i^+ > 0$:
\[
\lambda_i^+G_i(Y) \leq \lambda_i^+G_i(X) + \lambda_i^+\langle \nabla_{\mathrm{bw}}G_i(X), \log_X Y\rangle_X + \tfrac{\lambda_i^+}{2}W_2^2(X,Y).
\]
For each $j \in \mathcal{J}$, we apply Lemma~\ref{lemma:lower} with $\Sigma = \Sigma_j$ and multiply by $\lambda_j^- > 0$:
\[
-\lambda_j^-G_j(Y) \leq -\lambda_j^-G_j(X) - \lambda_j^-\langle \nabla_{\mathrm{bw}}G_j(X), \log_X Y\rangle_X + \tfrac{\lambda_j^-}{2}\big(\lambda_{\max}(T_{\Sigma_j}) - 1\big)_+W_2^2(X,Y).
\]
We sum these inequalities over $i \in \mathcal{I}$ and $j \in \mathcal{J}$. Using $\sum_{i\in\mathcal{I}}\lambda_i^+ = \mu_+$, we obtain
\begin{align}\label{eq:Phi-local}
\Phi(Y) \leq \Phi(X) &+ \langle \nabla_{\mathrm{bw}}\Phi(X), \log_X Y\rangle_X \nonumber\\
&+ \tfrac{1}{2}\Big(\mu_+ + \sum_{j\in\mathcal{J}}\lambda_j^-\big(\lambda_{\max}(T_{\Sigma_j}) - 1\big)_+\Big)W_2^2(X,Y).
\end{align}
Now assume that $\lambda_{\min}(X) \geq m$. For each $j \in \mathcal{J}$, Lemma~\ref{lemma:lower} gives
\[
\lambda_{\max}(T_{\Sigma_j}) \leq \sqrt{\lambda_{\max}(\Sigma_j)/\lambda_{\min}(X)} \leq \sqrt{\lambda_{\max}(\Sigma_j)/m}.
\]
Since $a \mapsto (a - 1)_+$ is non-decreasing, $(\lambda_{\max}(T_{\Sigma_j}) - 1)_+ \leq (\sqrt{\lambda_{\max}(\Sigma_j)/m} - 1)_+$. Hence, the coefficient of $\frac{1}{2}W_2^2(X,Y)$ in~\eqref{eq:Phi-local} is at most $L_m$. Since $W_2^2(X,Y) \geq 0$, the claim follows.

Finally, $\sum_{k=1}^n\lambda_k = 1$ gives $\mu_+ = 1 + \mu_-$, so $L_m \geq \mu_+ \geq 1$. If $\mathcal{J} = \emptyset$, then $L_m = \mu_+ = 1$.
\end{proof}

Finally, we present the proof of the convergence rate for Algorithm \ref{alg:bwgd}.

\begin{proof}[\textbf{Proof for Theorem \ref{thm:convergence-rate-bwgd}}]
First, we show that $F_\partial$ and $m_0$ are well defined. Indeed, the singular matrices form a closed subset of $\mathbb{S}_+^d$. Since $F$ is continuous and coercive on $\mathbb{S}_+^d$ (Step~1 of the proof of Theorem~\ref{thm:existence}), $F$ attains its minimum $F_\partial$ on this subset, and every minimizer of $F$ over $\mathbb{S}_+^d$ is nonsingular (Step~2). Hence, $F_\partial > F_\ast$. Moreover, the set $\mathcal{L} := \lbrace X \in \mathbb{S}_+^d : F(X) \leq F(S_0)\rbrace$ is compact by coercivity, and it contains no singular matrix since $F(S_0) < F_\partial$. Thus, $\mathcal{L} \subset \mathbb{S}_{++}^d$, and the continuous function $\lambda_{\min}$ attains a minimum $m_0>0$ on $\mathcal{L}$.

Next, we will show that $S_t \in \mathcal{L}$ and $\lambda_{\min}(S_t) \geq m_0$, for all $t$, by induction. Assume $S_t \in \mathcal{L}$, which holds for $t=0$, so that $\lambda_{\min}(S_t) \geq m_0$. Using $S_{t+1} = \mathrm{Exp}_{S_t}(-\eta \nabla_{\mathrm{bw}} \Phi(S_t))$ and Lemma~\ref{lemma:lsmooth}, we have
    \begin{align*}
        \Phi(S_{t+1}) &\leq \Phi(S_{t}) + \langle \nabla_{\mathrm{bw}} \Phi(S_t), \log_{S_t}(S_{t+1}) \rangle_{S_t} + \frac{L_{m_0}}{2}W^2_2(S_t,S_{t+1}) \\
        &= \Phi(S_{t}) + \langle \nabla_{\mathrm{bw}} \Phi(S_t), -\eta \nabla_{\mathrm{bw}} \Phi(S_t) \rangle_{S_t} + \frac{L_{m_0}}{2} \eta^2 \lVert \nabla_{\mathrm{bw}} \Phi(S_t) \rVert^2_{S_t} \\
        &= \Phi(S_t) - \eta\left( 1-\frac{\eta L_{m_0}}{2} \right)\lVert \nabla_{\mathrm{bw}} \Phi(S_t) \rVert^2_{S_t} \\
        &\leq \Phi(S_t) - \frac{\eta}{2} \lVert \nabla_{\mathrm{bw}} \Phi(S_t) \rVert^2_{S_t} \quad \text{(Since $\eta \leq  1/L_{m_0}$)}.
    \end{align*}
Thus, $F(S_{t+1}) \leq F(S_t) \leq F(S_0)$, and $S_{t+1} \succ 0$, so $S_{t+1} \in \mathcal{L}$. Therefore, by induction, $S_t \in \mathcal{L}$ and $\lambda_{\min}(S_t) \geq m_0$ for all $t$. Proceeding this over $t \in \lbrace 0,...,T-1\rbrace$ and using $\Phi(S_T) \geq F_\ast/2$,
$$\frac{\eta}{2}\sum_{t=0}^{T-1} \lVert \nabla_{\mathrm{bw}} \Phi(S_{t}) \rVert^2_{S_t} \leq \Phi(S_0) - \Phi(S_T) \leq \frac{F(S_0) - F_\ast}{2},$$
which is the first bound. The bound for $F$ follows from $\nabla_{\mathrm{bw}}F = 2\nabla_{\mathrm{bw}}\Phi$.
\end{proof}

\section{Spectral Dominance condition validation in two numerical experiments}\label{appen:condition-validation}
We would like to provide the exact numerical verification of the Spectral Dominance condition for the Ant Social Network dataset and DTI helical dataset. As defined in Theorem~\ref{thm:existence}, the condition requires that the positive weight terms (LHS) strictly dominate the negative weight terms (RHS):
\begin{itemize}
    \item LHS (Positive Terms): $\sum_{i \in \mathcal{I}} \lambda_i^+ \sqrt{\lambda_{\min}(\Sigma_i)}$
    \item RHS (Negative Terms): $\sum_{j \in \mathcal{J}} \lambda_j^- \sqrt{\lambda_{\max}(\Sigma_j)}$
\end{itemize}

\subsection{Ant Social Organization Networks}

\textbf{Graph-to-SPD representation.} The observed graph, equivalently its Laplacian $L$, is the ground-truth network object. For a connected graph, $L$ is positive semidefinite with $\mathrm{null}(L) = \mathrm{span}\{\mathbf{1}\}$, so it is singular and cannot be used directly as a point of $\mathbb{S}_{++}^d$. The rank-one correction removes this unique null direction: $L + \frac{1}{d}\mathbf{1}\mathbf{1}^\top$ is positive definite, and its inverse $\Sigma = L^\dagger + \frac{1}{d}\mathbf{1}\mathbf{1}^\top$ is the SPD representation used for BW regression, following~\cite{haasler2024bures}. A disconnected graph would have several zero Laplacian eigenvalues, and the rank-one correction would leave the matrix singular; this is why we select the first 11 days, on which all networks are connected. After regression, a prediction $\hat\Sigma$ is mapped back to the graph domain by $\hat L = \hat\Sigma^{-1} - \frac{1}{d}\mathbf{1}\mathbf{1}^\top$ and compared with the observed network.

The SPD matrix representation of each network admits two natural choices: 
the \textbf{unnormalized} regularized Laplacian $\Sigma := L^\dagger + 
\frac{1}{d}\mathbf{1}\mathbf{1}^\top$ and the \textbf{normalized} 
regularized Laplacian $\Sigma := L_{\mathrm{sym}}^\dagger + 
\frac{1}{d}\mathbf{1}\mathbf{1}^\top$, where $L_{\mathrm{sym}} := 
D^{-1/2}LD^{-1/2}$ is the symmetric normalized Laplacian and $D$ is the degree matrix. The main experiment in Subsection~\ref{subsec:ant-experiment} (Figure~\ref{fig:ant_results}) uses the unnormalized representation. The normalized representation is reported only as an additional alternative. We report 
the condition check for both representations and discuss their 
implications.

\begin{table}[!ht]
\centering
\begin{tabular}{|c|c|c|c|}
\hline
Predictor $\tau$ & LHS & RHS & LHS $>$ RHS \\
\hline
1 & 0.0357 & 0.3300 & False \\
2 & 0.0316 & 0.1629 & False \\
3 & 0.0288 & 0.0492 & False \\
4 & 0.0262 & 0.0000 & True \\
5 & 0.0264 & 0.0000 & True \\
6 & 0.0261 & 0.0000 & True \\
7 & 0.0258 & 0.0000 & True \\
8 & 0.0252 & 0.0000 & True \\
9 & 0.0262 & 0.0492 & False \\
10 & 0.0292 & 0.1629 & False \\
11 & 0.0332 & 0.3300 & False \\
\hline
\end{tabular}
\vspace{0.3cm}
\caption{Spectral Dominance condition verification for the Ant Social 
Organization Network using the unnormalized graph Laplacian 
$\Sigma := L^\dagger + \frac{1}{d}\mathbf{1}\mathbf{1}^\top$ (main experiment)}.
\label{tab:spectral_dominance_laplacian}
\end{table}
\vspace{-0.3cm}
\paragraph{Unnormalized Laplacian.}
Table~\ref{tab:spectral_dominance_laplacian} reports the Spectral 
Dominance condition for the unnormalized Laplacian. The condition is 
satisfied during interpolation days ($\tau \in \{4,\ldots,8\}$), where all regression weights are non-negative so that RHS $= 0$ trivially. 
However, the condition fails during extrapolation days 
($\tau \in \{1,2,3,9,10,11\}$). Since Condition~\eqref{eq:SNM} is sufficient but not necessary, this failure only means that the spectral certificate is inconclusive. It does not imply that the objective lacks a minimizer or that the update becomes indefinite. The condition replaces the full matrix geometry with scalar worst-case bounds and discards information about the relative eigenspaces of the matrices. Here, its failure is a consequence of the eigenvalue scaling of large networks: 
for a graph with $d = 113$ nodes, the maximum eigenvalue of the 
unnormalized Laplacian $L$ scales as $O(d)$, producing an eigenvalue 
spread of order $10^3$ between $\lambda_{\min}$ and $\lambda_{\max}$ 
of the regularized matrices. As a result, the RHS term, which involves 
$\sqrt{\lambda_{\max}(\Sigma_j)} \approx 1$, dominates the LHS term, 
which involves the much smaller $\sqrt{\lambda_{\min}(\Sigma_i)} 
\approx 0.025$. Consequently, the convergence reported in Figure~\ref{fig:BWGD-grad-norm} is theoretically certified on the interpolation days and empirical on the extrapolation days, where the iterates and congruence factors remained positive definite, and the gradient norm decreased below $10^{-10}$. Obtaining a tight, necessary-and-sufficient feasibility condition remains an open problem.

\begin{table}[h!]
\centering
\begin{tabular}{|c|c|c|c|}
\hline
Predictor $\tau$ & LHS & RHS &  LHS > RHS \\
\hline
1 & 1.1886 & 0.6650 & True \\
2 & 1.0443 & 0.3311 & True \\
3 & 0.9413 & 0.1012 & True \\
4 & 0.8961 & 0.0000 & True \\
5 & 0.8947 & 0.0000 & True \\
6 & 0.8925 & 0.0000 & True \\
7 & 0.8903 & 0.0000 & True \\
8 & 0.8891 & 0.0000 & True \\
9 & 0.9319 & 0.0963 & True \\
10 & 1.0325 & 0.3225 & True \\
11 & 1.1864 & 0.6594 & True \\
\hline
\end{tabular}
\vspace{0.3cm}
\caption{Spectral Dominance condition for the normalized regularized 
Laplacian $\Sigma := L_{\mathrm{sym}}^\dagger + \frac{1}{d}\mathbf{1}\mathbf{1}^\top$, with $L_{\mathrm{sym}} = D^{-1/2}LD^{-1/2}$ (alternative representation). The condition is strictly satisfied for all 
11 days, including extrapolation days.}
\label{tab:spectral_dominance_normalized}
\vspace{-0.5cm}
\end{table}
\paragraph{Normalized Laplacian.}
The normalized Laplacian $L_{\mathrm{sym}}$ resolves the eigenvalue scaling 
issue by construction: its eigenvalues lie in $[0, 2]$ regardless of 
network size~\cite{chung1997spectral}, eliminating the $O(d)$ growth 
that causes the condition to fail for large networks. As shown in 
Table~\ref{tab:spectral_dominance_normalized}, the Spectral Dominance 
condition is strictly satisfied for all 11 days, including all 
extrapolation days.

\subsection{Diffusion Tensor Imaging (DTI)} 
\vspace{-0.3cm}
\begin{table}[h!]
\centering
\begin{tabular}{|c|c|c|c|}
\hline
\textbf{Predictor Index} & \textbf{LHS} & \textbf{RHS} & \textbf{ LHS > RHS} \\
\hline
0 & 0.7405 & 0.6557 & True \\
5,000 & 0.6036 & 0.3496 & True \\
10,000 & 0.5115 & 0.1438 & True \\
15,000 & 0.5184 & 0.1592 &  True \\
20,000 & 0.5453 & 0.2194 &  True \\
25,000 & 0.5534 & 0.2374 &  True \\
30,000 & 0.5568 & 0.2450 &  True \\
35,000 & 0.5539 & 0.2385 &  True \\
40,000 & 0.5438 & 0.2160 &  True \\
45,000 & 0.5416 & 0.2111 &  True \\
50,000 & 0.5447 & 0.2180 &  True \\
55,000 & 0.5416 & 0.2111 &  True \\
60,000 & 0.5438 & 0.2160 &  True \\
65,000 & 0.5539 & 0.2385 &  True \\
70,000 & 0.5568 & 0.2450 &  True \\
75,000 & 0.5534 & 0.2374 &  True \\
80,000 & 0.5453 & 0.2194 &  True \\
85,000 & 0.5184 & 0.1592 &  True \\
90,000 & 0.5115 & 0.1438 &  True \\
95,000 & 0.6036 & 0.3497 &  True \\
\hline
\end{tabular}
\vspace{0.3cm}
\caption{Spectral Dominance condition verification for the large-scale 
DTI helical dataset ($n = 100{,}000$ tensors) at 20 uniformly spaced 
predictor locations. Despite negative weights being present at every 
evaluation point, the condition LHS $>$ RHS is strictly satisfied 
throughout, confirming that the BW Fréchet regression problem is 
well-posed across the entire helical trajectory.}
\label{table:DTI-validation}
\vspace{-0.3cm}
\end{table}
Regarding the large-scale DTI dataset ($n = 100,000$ sample tensors), we evaluated the Spectral Dominance condition for all 20 target predictor locations in Table~\ref{table:DTI-validation} (indices uniformly spaced from 0 to 95,000). The predictors are their 3-D coordinates $(x,y,z)$. Because this is a massive dataset with wide spatial dispersion, negative weights are present at every single evaluation point ($\text{RHS} > 0$). However, as shown in Table~\ref{table:DTI-validation}, the strictly positive stabilizing influence (LHS) safely bounds the negative terms (RHS) in every single instance.

\section{Wall-Clock Time: Full-Batch RGD vs.\ Pairwise 
R-SGD}\label{app:wallclock}
\vspace{-0.3cm}
Table~\ref{tab:wallclock} reports wall-clock times for regressing 
a single target point from $n = 100{,}000$ sample tensors over 
100 iterations, measured on an Apple MacBook Air M4 (16GB RAM). 
Full-batch Algorithm~1 incurs a cost of $\approx$4s per iteration 
since each step computes $n$ geometric means over the entire dataset. 
Pairwise Algorithm~2 reduces this to $\approx$0.07s per iteration 
by sampling a single pair $(i,j)$ at each step, yielding a 
57$\times$ overall speedup. This directly confirms the 
theoretical prediction that the per-iteration complexity of 
Algorithm~2 is independent of $n$.
\vspace{-0.3cm}
\begin{table}[h]
\centering
\begin{tabular}{lcc}
\toprule
\textbf{Algorithm} & \textbf{Per iteration} & 
\textbf{Total (100 iter.)} \\
\midrule
Full-batch RGD (Algorithm~1)  & $\approx 4$s   & $\approx 400$s \\
Pairwise R-SGD (Algorithm~2)  & $\approx 0.07$s & $\approx 7$s  \\
\bottomrule
\end{tabular}
\vspace{0.2cm}
\caption{Wall-clock time comparison for regressing one target 
point from $n=100{,}000$ tensors over 100 iterations. 
Hardware: Apple MacBook Air M4, 16GB RAM.}
\label{tab:wallclock}
\end{table}

%%%%%%%%%%%%%%%%%%%%%%%%%%%%%%%%%%%%%%%%%%%%%%%%%%%%%%%%%%%%
\newpage
\section*{NeurIPS Paper Checklist}
\begin{enumerate}
\item {\bf Claims}
    \item[] Question: Do the main claims made in the abstract and introduction accurately reflect the paper's contributions and scope?
    \item[] Answer: \answerYes{}
    \item[] Justification: The abstract and introduction state the theoretical contributions: the Spectral Dominance existence condition (Theorem~\ref{thm:existence}), the landscape characterization free of local maxima (Proposition~\ref{prop:no-maximum}), local uniqueness under additional assumptions (Proposition~\ref{thm:uniqueness}), and projection-free algorithms (Proposition~\ref{prop:non-projection-iteration} and Corollary~\ref{col:projection-free-stochastic}) with convergence rates (Theorems~\ref{thm:convergence-rate-bwgd} and~\ref{thm:RSGD-convergence}), all formally established in the paper. Empirical claims are substantiated in Section~\ref{sec:num-experiment}: BW outperforms the baselines on network topology metrics (Figures~\ref{fig:ant_results} and~\ref{fig:network-plot}), and the stochastic algorithm achieves a $57\times$ wall-clock speedup on $n=100{,}000$ tensors while remaining stable across random pairwise sampling and maintaining accurate predictions (Figure~\ref{fig:obj_value}, Appendix~\ref{subsec:appen-DTI}, Appendix~\ref{app:wallclock}).

\item {\bf Limitations}
    \item[] Question: Does the paper discuss the limitations of the work performed by the authors?
    \item[] Answer: \answerYes{}
    \item[] Justification: Remark~\ref{rem:uniqueness-vs-existence}, Section~\ref{sec:RGD}, and Section~\ref{sec:conclusion} discuss limitations directly:
    \begin{itemize}
        \item[(1)] The Spectral Dominance condition is sufficient but not necessary, leaving the tightness of the feasibility boundary as an open problem (Section~\ref{sec:conclusion}).
        \item[(2)] The uniqueness result (Proposition~\ref{thm:uniqueness}) is local and requires strictly stronger assumptions than existence (Theorem~\ref{thm:existence}), namely co-location of all data matrices in a common geodesic ball (Remark~\ref{rem:uniqueness-vs-existence}).
        \item[(3)] The convergence guarantees rely on additional conditions: Theorem~\ref{thm:convergence-rate-bwgd} requires an initialization in a sublevel set, Theorem~\ref{thm:RSGD-convergence} assumes a uniform lower eigenvalue bound on the iterates, and the pairwise condition~\eqref{eqn:Stricter-existence-condition} for Algorithm~\ref{alg:rsgd} is stronger than~\eqref{eq:SNM} (Section~\ref{sec:RGD}).
        \item[(4)] The pairwise stochastic algorithm scales efficiently in sample size $n$, but scaling to very large matrix dimension $d$ remains a direction for future work (Section~\ref{sec:conclusion}).
    \end{itemize}

\item {\bf Theory assumptions and proofs}
    \item[] Question: For each theoretical result, does the paper provide the full set of assumptions and a complete (and correct) proof?
    \item[] Answer: \answerYes{}
    \item[] Justification: All assumptions are explicitly stated in each result. Complete proofs are provided in Appendix~\ref{appen:proof} for Theorem~\ref{thm:existence}, Propositions~\ref{prop:stationary-bound}, \ref{prop:no-maximum}, and~\ref{prop:non-projection-iteration}, Lemmas~\ref{prop:upper-bound-curvature}, \ref{lemma:lower}, and~\ref{lemma:lsmooth}, and Theorem~\ref{thm:convergence-rate-bwgd}. Corollary~\ref{col:projection-free-stochastic} follows from the proof of Proposition~\ref{prop:non-projection-iteration}, and Lemma~\ref{lemma:upper} is the $1$-smoothness inequality of~\cite{chewi2020gradient}. Proposition~\ref{thm:uniqueness} and Theorem~\ref{thm:RSGD-convergence} apply \cite[Theorem~3.4.9]{wintraecken2015ambient} and~\cite{oowada2025faster}, respectively, and all their assumptions are stated explicitly. All theorems and lemmas are numbered and cross-referenced.

\item {\bf Experimental result reproducibility}
    \item[] Question: Does the paper fully disclose all the information needed to reproduce the main experimental results of the paper to the extent that it affects the main claims and/or conclusions of the paper (regardless of whether the code and data are provided or not)?
    \item[] Answer: \answerYes{}
    \item[] Justification: Subsection~\ref{subsec:ant-experiment} specifies the ant social network experiment completely: dataset source (Network Repository~\cite{nr}), Laplacian regularization procedure, initial point ($S_0 = I_{d\times d}$), step size ($\eta = 1/\sum_k|\lambda_k|$), and iteration count ($T=100$). Subsection~\ref{subsec:DTI} specifies the DTI experiment: synthetic data generation (helical backbone coordinates and tensor orientations), learning rate schedule ($\eta_t = \eta_0/\sqrt{t+1}$, $\eta_0 = 1$), number of independent runs (10), and target indices. Hardware is reported in Appendix~\ref{app:wallclock}.

\item {\bf Open access to data and code}
    \item[] Question: Does the paper provide open access to the data and code, with sufficient instructions to faithfully reproduce the main experimental results, as described in supplemental material?
    \item[] Answer: \answerNo{}
    \item[] Justification: Code is not included with the paper. However, the ant social network dataset is publicly available via the Network Repository~\cite{nr}, and the DTI dataset is fully synthetic, with the generation procedure described in Subsection~\ref{subsec:DTI}. The algorithms are specified completely in Algorithms~\ref{alg:bwgd} and~\ref{alg:rsgd}, enabling independent reimplementation. Code will be released publicly.

\item {\bf Experimental setting/details}
    \item[] Question: Does the paper specify all the training and test details (e.g., data splits, hyperparameters, how they were chosen, type of optimizer) necessary to understand the results?
    \item[] Answer: \answerYes{}
    \item[] Justification: All hyperparameters are reported in Section~\ref{sec:num-experiment}: the step size $\eta = 1/\sum_k|\lambda_k|$ for Algorithm~\ref{alg:bwgd}, which equals the smoothness constant of Lemma~\ref{lemma:lsmooth} when all weights are non-negative, the initial points, the iteration counts, and the diminishing learning rate schedule $\eta_t = \eta_0/\sqrt{t+1}$ with $\eta_0 = 1$ for Algorithm~\ref{alg:rsgd}.

\item {\bf Experiment statistical significance}
    \item[] Question: Does the paper report error bars suitably and correctly defined or other appropriate information about the statistical significance of the experiments?
    \item[] Answer: \answerYes{}
    \item[] Justification: Figure~\ref{fig:obj_value} (DTI experiment) reports the mean $\pm 1$ standard deviation of objective values across 10 independent runs per target point, with the source of variability (random pairwise sampling in Algorithm~\ref{alg:rsgd}) explicitly stated in Subsection~\ref{subsec:DTI}. The ant network experiment uses the deterministic Algorithm~\ref{alg:bwgd}, so error bars are not applicable; convergence is instead reported as gradient norm trajectories in Figure~\ref{fig:BWGD-grad-norm}.

\item {\bf Experiments compute resources}
    \item[] Question: For each experiment, does the paper provide sufficient information on the computer resources (type of compute workers, memory, time of execution) needed to reproduce the experiments?
    \item[] Answer: \answerYes{}
    \item[] Justification: Appendix~\ref{app:wallclock} reports the hardware used for the DTI wall-clock experiment (Apple MacBook Air M4, 16GB RAM) along with per-iteration and total runtimes for both algorithms ($\approx$4s/iter for Algorithm~\ref{alg:bwgd}, $\approx$0.07s/iter for Algorithm~\ref{alg:rsgd}). The ant network experiment is small-scale (11 networks, 113 nodes each, 100 iterations) and completes on the same hardware.

\item {\bf Code of ethics}
    \item[] Question: Does the research conducted in the paper conform, in every respect, with the NeurIPS Code of Ethics \url{https://neurips.cc/public/EthicsGuidelines}?
    \item[] Answer: \answerYes{}
    \item[] Justification: This paper presents foundational mathematical research on Riemannian optimization and conditional barycenters. No human subjects, sensitive personal data, or dual-use risks are involved. The ant social network dataset involves ant colonies (non-human subjects) and is publicly available. All prior work is properly cited.

\item {\bf Broader impacts}
    \item[] Question: Does the paper discuss both potential positive societal impacts and negative societal impacts of the work performed?
    \item[] Answer: \answerNA{}
    \item[] Justification: This is foundational theoretical research on mathematical optimization over SPD matrices with no direct path to negative societal impacts. The work is several abstraction layers removed from any deployment. Potential downstream applications (improved DTI analysis in neuroimaging, network regression in biology) are beneficial, and no negative use cases are readily identifiable from this mathematical framework.

\item {\bf Safeguards}
    \item[] Question: Does the paper describe safeguards that have been put in place for responsible release of data or models that have a high risk for misuse (e.g., pre-trained language models, image generators, or scraped datasets)?
    \item[] Answer: \answerNA{}
    \item[] Justification: The paper does not release pretrained models, scraped datasets, or any assets with high risk for misuse. The synthetic DTI dataset is procedurally generated by the authors. The ant social network is a pre-existing public dataset of ant interaction data.

\item {\bf Licenses for existing assets}
    \item[] Question: Are the creators or original owners of assets (e.g., code, data, models), used in the paper, properly credited and are the license and terms of use explicitly mentioned and properly respected?
    \item[] Answer: \answerYes{}
    \item[] Justification: The ant social network dataset is cited to both the original collection paper~\cite{mersch2013tracking} and the Network Repository~\cite{nr}, which provides datasets for academic research use. All algorithmic tools and theoretical results from prior work are properly cited throughout the paper.

\item {\bf New assets}
    \item[] Question: Are new assets introduced in the paper well documented and is the documentation provided alongside the assets?
    \item[] Answer: \answerNA{}
    \item[] Justification: The paper does not release new standalone datasets, models, or code assets. The synthetic DTI dataset is generated from the procedure described in Subsection~\ref{subsec:DTI} and is not packaged as a separate asset.

\item {\bf Crowdsourcing and research with human subjects}
    \item[] Question: For crowdsourcing experiments and research with human subjects, does the paper include the full text of instructions given to participants and screenshots, if applicable, as well as details about compensation (if any)?
    \item[] Answer: \answerNA{}
    \item[] Justification: The paper involves no crowdsourcing or human subjects. The ant social network dataset was collected from ant colonies (non-human subjects) by the original authors of~\cite{mersch2013tracking}.

\item {\bf Institutional review board (IRB) approvals or equivalent for research with human subjects}
    \item[] Question: Does the paper describe potential risks incurred by study participants, whether such risks were disclosed to the subjects, and whether Institutional Review Board (IRB) approvals (or an equivalent approval/review based on the requirements of your country or institution) were obtained?
    \item[] Answer: \answerNA{}
    \item[] Justification: No human subjects are involved in this research. IRB approval is not required.

\item {\bf Declaration of LLM usage}
    \item[] Question: Does the paper describe the usage of LLMs if it is an important, original, or non-standard component of the core methods in this research? Note that if the LLM is used only for writing, editing, or formatting purposes and does \emph{not} impact the core methodology, scientific rigor, or originality of the research, declaration is not required.
    \item[] Answer: \answerNA{}
    \item[] Justification: LLMs are not used as any component of the core methodology, theoretical analysis, experimental design, or results of this paper. Any LLM assistance was strictly limited to writing and editing purposes, which does not require declaration per NeurIPS policy.
\end{enumerate}

\end{document}